\documentclass[12pt]{article}
\usepackage[linesnumbered,ruled,vlined]{algorithm2e}
\usepackage{algpseudocode}
\usepackage{amsmath}
\usepackage{amsfonts}
\usepackage{amssymb}
\usepackage{amsthm}
\usepackage{authblk}
\usepackage{bm}
\usepackage{booktabs} % http://ctan.org/pkg/booktabs
\usepackage[font=scriptsize]{caption}
\usepackage{color}
\usepackage{esvect}
\usepackage{enumitem}
\usepackage{extarrows}
\usepackage{float}
\usepackage{footmisc} % provides \footref and \mpfootnotemark commands
\usepackage{graphicx}
\usepackage{hhline}
\usepackage[utf8]{inputenc}
\usepackage{hyperref} % see the file hyperref.cfg for options
\usepackage{indentfirst}
\usepackage{mathrsfs}
\usepackage{multirow}
\usepackage{modroman}
\usepackage{natbib}
\usepackage{geometry}
\usepackage{pdfpages}
\usepackage{setspace}
\usepackage{subcaption}
\usepackage{tabularx}
\usepackage{url}
\usepackage{xcolor}
\usepackage{xspace}

\newtheorem{theorem}{Theorem}
\newtheorem{lemma}{Lemma}
\newtheorem{prop}{Proposition}
\newtheorem{condition}{Condition}
\newtheorem{corollary}{Corollary}

\newtheorem{definition}{Definition}
\newtheorem{proposition}{Proposition}

\makeatletter
\newcommand{\distas}[1]{\mathbin{\overset{#1}{\kern\z@\sim}}}%
\newsavebox{\mybox}\newsavebox{\mysim}
\newcommand{\distras}[1]{%
  \savebox{\mybox}{\hbox{\kern3pt$\scriptstyle#1$\kern3pt}}%
  \savebox{\mysim}{\hbox{$\sim$}}%
  \mathbin{\overset{#1}{\kern\z@\resizebox{\wd\mybox}{\ht\mysim}{$\sim$}}}%
}
\makeatother

\newcommand{\tabitem}{~~\llap{\textbullet}~~}
\providecommand{\keywords}[1]
{
  \small	
  \textbf{\textit{Keywords:}} #1
}

\SetKw{KwMeta}{Meta-parameters:}

\numberwithin{equation}{section}
\numberwithin{theorem}{section}
\numberwithin{lemma}{section}
\numberwithin{prop}{section}
\numberwithin{condition}{section}
\numberwithin{corollary}{section}
\numberwithin{remark}{section}
\numberwithin{example}{section}
\numberwithin{definition}{section}

\title{Smoothing the Conditional Value-at-Risk based Pickands Estimators}
\author[b]{\normalsize Yizhou Li\thanks{Corresponding author at Department of Applied Mathematics and Statistics at Stony Brook University,
United States. E-mail address: \href{mailto:yizhou.li@stonybrook.edu}{\url{yizhou.li@stonybrook.edu}}}}
\author[a,b,c]{\normalsize Pawe\l\ Polak\thanks{E-mail address: \href{mailto:pawel.polak@stonybrook.edu}{\url{pawel.polak@stonybrook.edu}}}}

\affil[a]{\small \textit{Center of Excellence in Wireless and Information Technology, Stony Brook University, United States}}
\affil[b]{\small \textit{Department of Applied Mathematics and Statistics, Stony Brook University, United States}}
\affil[c]{\small \textit{Institute for Advanced Computational Science, Stony Brook University, United States}}

\setcitestyle{authoryear,open={(},close={)}} %Citation-related commands
\begin{document}
\maketitle

\begin{abstract}
We incorporate the conditional value-at-risk (CVaR) quantity into a generalized class of Pickands estimators. By introducing CVaR, the newly developed estimators not only retain the desirable properties of consistency, location, and scale invariance inherent to Pickands estimators, but also achieve a reduction in mean squared error (MSE). To address the issue of sensitivity to the choice of the number of top order statistics used for the estimation, and ensure robust estimation, which are crucial in practice, we first propose a beta measure, which is a modified beta density function, to smooth the estimator. Then, we develop an algorithm to approximate the asymptotic mean squared error (AMSE) and determine the optimal beta measure that minimizes AMSE.  A simulation study involving a wide range of distributions shows that our estimators have good and highly stable finite-sample performance and compare favorably with the other estimators.
\end{abstract}

\section{Introduction} \label{section:intro}
Suppose $X$ is a random variable (r.v.) with a distribution function (d.f.) $F$, and $X_{1}, X_{2}, \ldots, X_{n}$ is a sequence of independent samples of $X$. The extreme value d.f. with shape parameter $\gamma\in\mathbb{R}$ (shift parameter $0$ and scale parameter $1$) is defined by
\begin{align*}
    G_{\gamma}(x) = \begin{cases}
        \exp \left\{-(1+\gamma x)^{-1 / \gamma}\right\}, & \quad \text{for } \gamma\neq0,\ 1+\gamma x>0, \\
        \exp \left\{-e^{-x}\right\}, & \quad \text{for } \gamma=0,\ x\in\mathbb{R}.
    \end{cases}
\end{align*}
The d.f. $F$ belongs to the \textit{domain of attraction} of $G_\gamma$ for some $\gamma\in \mathbb{R}$, noted as $F\in\mathscr{D}(G_\gamma)$, if there exists a sequence of constants $a_n>0$ and $b_n \in \mathbb{R}$ such that
\begin{equation} \label{eqn: DoA}
    \lim _{n \rightarrow \infty} P\left\{\max \left(X_{1}, \ldots, X_{n}\right) \leqslant a_{n}x + b_{n}\right\}=\lim _{n \rightarrow \infty} F^{n}\left(a_{n} x+b_{n}\right)=G_{\gamma}(x).
\end{equation}
This $\gamma$ is called the \textit{extreme value index} of $F$.

An alternative characterization of $\gamma$ in extreme value theory can be provided using the generalized Pareto distribution (GPD), which has d.f.
\begin{align*}
    H_\gamma(x)=\begin{cases}
        1-(1+\gamma x)_+^{-1 / \gamma}, & \quad \text{for }\gamma\neq0, \ x\geq0, \\
        1 - e^{-x}, & \quad \text{for } \gamma=0, \ x\geq0,
    \end{cases}
\end{align*}
where $z_+=\max{(z,0)}$. Note that $H_\gamma$ with $\gamma>0$ yields a Pareto distribution; $H_0$ coincides with the exponential distribution. Additionally, $H_\gamma$ with $\gamma<0$ exhibits a finite right endpoint, particularly with $\gamma=-1$ resulting in a uniform distribution on the interval $(0,1)$.

Let $x_F=\sup\{x\in\mathbb{R}:F(x)<1\}$ denote the right endpoint of a d.f. $F$. For $u\in\mathbb{R}$ such that $F(u)<1$, let the excess d.f. denoted by $F_u(x)=P(X-u\leq x|X>u)$ for $x\geq0$. \cite{balkema1974residual} and \cite{pickands1975statistical} show that $F\in\mathscr{D}(G_\gamma)$ if and only if for $u<x_F$ there exists a positive, measurable scaling function $\sigma(u)$, such that
\begin{equation} \label{eqn: PoT}
    \lim_{u\to x_F} \sup_{x\geq0} \left|F_u(x) - H_\gamma\left(\frac{x}{\sigma(u)}\right)\right| = 0.
    % \lim_{u\to x_F} F_u(\sigma(u)x) = H_\gamma(x), \quad \text{for } x\geq0.
\end{equation}
Formula (\ref{eqn: PoT}) suggests modeling the unknown excess d.f. $F_u$ in the case of $F\in\mathscr{D}(G_\gamma)$ by the parametric family $H_\gamma(\cdot/\sigma)$, where $\gamma\in\mathbb{R}$ and $\sigma>0$. Consequently, the upper tail of $F$ can be approximated by a GPD. Moreover, if $F\in\mathscr{D}(G_\gamma)$ with $\gamma>0$, the function $\sigma(u)$ in \eqref{eqn: PoT} can be selected as $\sigma(u)=\gamma u$, and then $H_\gamma(x/(\gamma u)) = 1 - (u/(x+u))^{1/\gamma}$ for $x\geq0$ which is a Pareto distribution (see for example Lemma 4, \citealp{makarov2007applications}). 

Hence, understanding the extreme value index $\gamma$ is crucial for modeling maxima and estimating extreme quantiles. The estimation of $\gamma$, besides high quantile estimation, is one of the most crucial problems in univariate extreme value theory.

Compared with the \cite{hill1975simple} estimator and the probability-weighted moment estimator proposed by \cite{hosking1985estimation}, the primary benefits of the Pickands estimator by \cite{pickands1975statistical} include the invariant property under location and scale shift, and the consistency for any $\gamma \in \mathbb{R}$ under intermediate sequence such that $m_{n} \to \infty$, $m_{n} / n \to 0$ as $n\to\infty$. Let $X_1^{(n)},X_2^{(n)},\cdots,X_n^{(n)}$ denote the descending order statistics of $X_1,\ldots,X_n$. \cite{yun2002generalized} generalized the Pickands estimator by
\begin{equation} \label{eqn: yun}
    \widehat{\gamma}_{n, m}^{Y}(u, v):=\frac{1}{\log v} \log \frac{X_{m}^{(n)}-X_{[u m]}^{(n)}}{X_{[v m]}^{(n)}-X_{[u v m]}^{(n)}}, \quad u, v>0, u, v \neq 1,
\end{equation}
where $1 \leqslant m,[u m],[v m],[u v m] \leqslant n$, $[x]$ denotes the integer part of $x \in \mathbb{R}$, and $\lfloor x \rfloor$ is the floor of a real number $x$. \cite{teng2021thesis} introduced the Conditional Value at Risk (CVaR) order statistic, also known as the empirical super-quantile, into the Pickands estimator by replacing the descending order statistics in \eqref{eqn: yun} with the empirical descending CVaR order statistics $Y_{1}^{(n)} \geqslant Y_{2}^{(n)} \geqslant \cdots \geqslant Y_{n}^{(n)}$ by
\begin{equation} \label{eqn: cvar order statistics}
    Y_{k}^{(n)}:= \frac{1}{k} \sum_{i = 1}^{k} X_{i}^{(n)}, \quad k = 1,\ldots,n.
\end{equation}
From \cite{li2024asymptotic}, the asymptotic behavior of such estimator demonstrates a substantial reduction in asymptotic variance due to the equal-weighted averaging of CVaR \citep{rockafellar2002conditional} when the underlying distribution tail is not too heavy, i.e., $\gamma<1/2$. 

In practice, an estimator would be limited to be employed due to its sensitivity to the choice of the intermediate order statistics $\{X_{i}^{(n)}\}_{i=1}^m$. Such sensitivity is twofold. First, the inclusion of one single large-order statistic in estimation, i.e., incrementing $m$ by 1 for an estimator can considerably alter the estimated value. By incorporating CVaR order statistics, such an issue can be substantially mitigated. There are also many well-known estimators proposed to fix the problem and we denote them as \textquotedblleft smoothed estimators\textquotedblright such as the negative Hill estimator by \cite{falk1995some}, the smoothed Hill estimator by \cite{resnick1997smoothing}, the smoothed moment estimator by \cite{resnick1999smoothing}, the weighted least squares estimator by \cite{husler2006weighted}, kernel smoothed estimators by \cite{csorgo1985kernel} and \cite{groeneboom2003kernel}, the weighted mixture of Pickands estimators by \cite{drees1995refined}, the generalized Pickands estimator by \cite{segers2005generalized} who realized the Pickands estimator is essentially a linear combination of log-spacings of order statistics, and many more. 

Second, the performance of an estimator dramatically changes across different $m$. The number $m$ used in the implementation of an estimator depends strongly on the tail itself and the choice of this number is clearly a question of trade-off between bias and variance: as $m$ increases, the bias will grow because the tail satisfies less the convergence criterion, while if less data are used, the variance increases. This is why a plot of MSE of an estimator versus varying $m$ usually looks like a `U' shape. It is therefore suggested that the optimal value of $m$ should coincide with the value that minimizes the MSE. 

To address the aforementioned problem, we first introduce the CVaR-based smoothed estimator into the generalized class of the Pickands estimator by \cite{segers2005generalized}. The mixture of log-spacings follows weights based on a measure in a certain space satisfying several integral conditions. In addition, we propose a beta measure which is a modified beta density function that satisfies the integral conditions. Unless stated otherwise, all beta measures discussed in this paper refer to modified beta density according to these conditions. Such a structure provides flexibility to control bias and variance. One might aim to minimize the asymptotic mean squared error (AMSE) with respect to the beta measure. However, obtaining the exact value of AMSE is challenging, as it requires knowledge or a reliable estimate of a term from the second-order regular variation condition (see Sections \ref{sec: regular variation} and \ref{sec: measure selection}). Therefore, we propose a heuristic algorithm to approximate AMSE using a bootstrapping approach, allowing the optimal beta measure to be computed numerically. For numerical convenience, we also introduce an alternative objective called the regularized measure squared error (RegMSE), which is heuristically designed to be optimized with respect to beta measures. Minimizing both objectives ensures that our estimators maintain a low MSE while exhibiting a very stable (almost flat) performance across varying $m$. This stability makes the task of choosing $m$ very trivial in practice (see Section \ref{sec: simulations}).

The rest of the paper is organized as follows. In Section \ref{section:smooth}, we introduce the CVaR-based generalized class of Pickands estimator. Section \ref{sec: large sample} exhibits large sample properties of our CVaR-based smoothed estimator standing on needed conditions including weak consistency and asymptotic behavior. In Section \ref{sec: measure selection}, we show how to obtain optimal measures in the estimator based on asymptotic behavior. Simulation studies and conclusions are given in Section \ref{sec: simulations} and Section \ref{sec: conclusion}, respectively.

\section{Smoothing the Pickands Estimators} \label{section:smooth}
The concept of forming linear combinations of log-spacings of order statistics serves as the foundation for a broad family of extreme value index estimators. First, define $\Lambda$ as the collection of signed Borel measures $\lambda$ on the interval $(0,1]$ such that
\begin{equation} \label{eqn: space}
    \lambda(0,1]=0,\quad \int\log(1/t)|\lambda|(dt)<\infty,\quad \text{and} \quad \int\log(1/t)\lambda(dt)=1.
\end{equation}
If not mentioned explicitly, the domain of integration is always understood to be $(0,1]$. For $t\in[0,1]$, let $\lambda(t)=\lambda((0,t])$. Denote by $\lceil x \rceil$ the smallest integer at least as large as $x\in\mathbb{R}$, and, for convenience, set $\log(0):=0$. Then, the generalized Pickands estimator by \cite{segers2005generalized} is given by
\[
    \widehat{\gamma}_{n,m}(c,\lambda) = \int \log(X_{\lfloor c\lceil tm \rceil\rfloor}^{(n)} - X_{\lceil tm \rceil}^{(n)})\lambda(dt),
\]
where $m=1,\dots,n-1,\ 0<c<1$ and $\lambda\in\Lambda$. This family encompasses and refines the Pickands estimator and all its variants. In Section \ref{sec: simulations}, we verify in simulations that this estimator has a stable behavior as a function of $m$. Thus, we extend the CVaR order statistics to the same framework and propose the CVaR-based smoothed estimator. \\

\begin{definition}
For $m=1,\dots,n-1,\ 0<c<1$, and $\lambda\in\Lambda$, the CVaR-based smoothed estimator $\widehat{\gamma}_{n,m}(c,\lambda)$ is defined as
\[
    \widehat{\gamma}_{n,m}(c,\lambda) = \int \log(Y_{\lfloor c\lceil tm \rceil\rfloor}^{(n)} - Y_{\lceil tm \rceil}^{(n)})\lambda(dt).
\]
\end{definition}
\noindent The CVaR-based smoothed estimator is a linear combination of log-spacings of CVaR order statistics. The spacings are controlled by $c$, $m$, and the weights by $\lambda\in\Lambda$. In particular, let $j=\lceil tm \rceil$ and when $\frac{j-1}{m}<t\leq\frac{j}{m}$, we have
\begin{equation} \label{eqn: estimator}
    \widehat{\gamma}_{n,m}(c,\lambda) = \sum_{j=1}^m \{\lambda(\frac{j}{m})-\lambda(\frac{j-1}{m})\}\log(Y_{\lfloor c\lceil tm \rceil\rfloor}^{(n)} - Y_{\lceil tm \rceil}^{(n)}).
\end{equation}
Note that the CVaR-based Pickands estimator is a special case of \eqref{eqn: estimator}. For $0<c<1,\ 0<v<1$, and define $\lambda^{(v)}:=(\epsilon_1-\epsilon_v)/\log v$ where $\epsilon_x$ is the point-measure giving mass 1 to $x$, the CVaR-based Pickands estimator is recovered when assigning $c=u$ and $t=v$.

\section{Large Sample Properties}
\label{sec: large sample}
\subsection{Weak Consistency} \label{sec: weak consistency}
Let `$\xrightarrow{p}$' denote convergence in probability. A sequence of positive integers $m=m(n)$ is called an intermediate sequence if $1\leq m \leq n-1$, $m \to \infty$ and $m / n \to 0$ as $n \to \infty$.
\begin{theorem} \label{thm: weak consistency}
    Let $F\in\mathscr{D}(G_\gamma)$ with $\gamma<1$, $0<c<1$, and $\lambda\in\Lambda$. For every intermediate sequence $m$ we have $\widehat{\gamma}_{n,m}(c,\lambda) \xrightarrow{p} \gamma$ as $n\to\infty$.
\end{theorem}
\subsection{The CVaR-based Regular Variation} \label{sec: regular variation}
The tail quantile function $U$ of a d.f. $F$ is given by
\[
    U(x)=\begin{cases}
        F^{-1}(1-1/x), \quad x>1, \\
        0, \quad 0<x\leq1.
    \end{cases}
\]
The condition $F\in\mathscr{D}(G_\gamma)$ is equivalent to the existence of a positive, measurable function $a(t)$ such that
\begin{equation} \label{eqn: 1st_order_ERV_U}
    \lim_{t\to\infty} \frac{U(\frac{t}{y})-U(t)}{a(t)}=\frac{y^{-\gamma}-1}{\gamma} := h_{\gamma}(y), \quad y>0,
\end{equation}
where $h_{\gamma}(y)$ has to be read as $\log{\frac{1}{y}}$ in case $\gamma=0$. Moreover, (\ref{eqn: DoA}) holds with $b_n:=U(n)$ and $a_n:=a(n)$ (see Lemma 1 in \citealp{de1984slow}, and Theorem 1.1.6 in \citealp{de2006extreme}). The auxiliary function $a$ in (\ref{eqn: 1st_order_ERV_U}) is regularly varying with index $\gamma$, denoted as $a\in RV_\gamma$, in other words, 
\begin{equation}\label{eqn: 1st_order_RV}
    \lim_{t\to\infty} \frac{a(\frac{t}{y})}{a(t)}=y^{-\gamma},\quad y>0.
\end{equation}
If $a\in RV_0$, $a$ is called a \textit{slowly varying} function. Both \eqref{eqn: 1st_order_ERV_U} and \eqref{eqn: 1st_order_RV} hold locally uniformly in $0<y<\infty$ \citep[][Thereom $1.5.2$]{bingham1989regular}.

Suppose $F$ has a finite mean and the tail super-quantile function $V$ of $F$ is given by
\begin{equation} \label{eqn:V}
    V(x) = \frac{1}{x} \int_{0}^{x} F^{-1}( 1- s) ds = \frac{1}{x} \int_{0}^{x} U(\frac{1}{s}) ds, \quad 0<x<1.
\end{equation}
By \cite{teng2021thesis}, the condition $F\in\mathscr{D}(G_\gamma)$ implies that there exists a positive, measurable function $a(t)$ from \eqref{eqn: 1st_order_ERV_U} such that
\begin{equation} \label{eqn:1st_order_ERV'_V}
\lim_{t\to\infty} \frac{V(\frac{y}{t})-V(\frac{1}{t})}{a(t)} = \frac{1}{y} \int_0^y h_\gamma(w) dw - \int_0^1 h_\gamma(w) dw := \tilde{h}_\gamma(y) \quad \text { for } y>0,
\end{equation}
where $\tilde{h}_\gamma(y) = \frac{y^{-\gamma}-1}{\gamma(1-\gamma)}$ for $\gamma<1$ and $\tilde{h}_\gamma(y) = -log(y)$ in case $\gamma=0$, and the convergence hold locally uniformly in $0<y<\infty$.

The second-order regular variation condition is needed to establish the asymptotic normality of our estimator.
\begin{condition} \label{cond1}
 For some $t_0>1$ the tail quantile function $U$ is absolutely continuous on $[t_0,\infty)$ with density $u$. There are $\gamma\in\mathbb{R}, \ \rho\leq0, d\in\mathbb{R}$, and $A\in RV_{\rho}$ with $\lim_{t\to\infty}A(t)=0$ such that, denoting $a(t)=t u(t)$, we have
\begin{equation} \label{eqn:cond_alternative}
    \lim_{t\to\infty} \frac{\log a(\frac{t}{y}) - \log a(t) - \gamma\log\frac{1}{y}}{A(t)} = dh_\rho(y), \quad \text{for } y>0,
\end{equation}
and $a(t)$ is regularly varying with index $\gamma$.
\end{condition}

Define that
\begin{equation} \label{eqn:h}
h(y;t) = \frac{U(\frac{t}{y}) - U(t)}{a(t)} \quad \text{for }y>0, \ t\geq t_0,
\end{equation}
\begin{equation} \label{eqn:H}
H_{\gamma,\rho}(y) = \int_y^1 w^{-1-\gamma}h_\rho(w) dw = \frac{1}{\rho}\biggl(\frac{y^{-(\rho+\gamma)}-1}{\rho+\gamma} - \frac{y^{-\gamma}-1}{\gamma}\biggr), \quad \text{for }y>0,
\end{equation}
with the appropriate limits in case $\gamma=0,\ \rho=0$, or $\gamma+\rho=0$. Note that \eqref{eqn: 1st_order_ERV_U} can be reformulated as $\lim_{t\to\infty}h(y;t)=h_\gamma(y)$. Condition \ref{cond1} imposes a rate of convergence in \eqref{eqn: 1st_order_ERV_U} and implies
\begin{equation} \label{eqn:2nd_order_ERV}
    \lim_{t\to\infty} \frac{h(y;t)-h_\gamma(y)}{A(t)} = dH_{\gamma,\rho}(y)
\end{equation}
\noindent which is consistent with the framework of second-order generalized variation (see \citealp{de1996generalized} and Section 2.3, \citealp{de2006extreme}).

\begin{condition}
\label{cond2}
Let $A$ and $d$ be as in condition \ref{cond1}. The intermediate sequence $m=m(n)$ satisfies 
\begin{itemize}
    \item case $d\neq0$: $\lim_{n\to\infty} \sqrt{m} A(\frac{n}{m})$ exists;
    \item case $d=0$: $\sup_{n\geq2} \sqrt{m} A(\frac{n}{m}) < \infty$.
\end{itemize}
In both cases, we denote $r:=\lim_{n\to\infty} d \sqrt{m} A(\frac{n}{m})$.
\end{condition}

The second-order regular variation limit function for the tail super-quantile function $V$ is similar. Define that
\begin{equation} \label{eqn:h_tilde}
\tilde{h}_(y;t) := \frac{1}{y}\int_0^y h(w;t) dw - \int_0^1  h(w;t) dw = \frac{V(\frac{y}{t}) - V(\frac{1}{t})}{a(t)}, \quad y>0 \text{ and } t\geq t_0.
\end{equation}
By \cite{li2024asymptotic}, if \eqref{eqn:2nd_order_ERV} holds with $\gamma<1$, it can be implied that
\begin{equation}
\label{eqn:2nd_order_ERV_V}
    \lim_{t\to\infty} \frac{\tilde{h}_(y;t) - \tilde{h}_\gamma(y)}{A(t)}=d\tilde{H}_{\gamma,\rho}(y), \quad y>0,
\end{equation}
where
\begin{equation} \label{eqn:H_tilde}
    \tilde{H}_{\gamma,\rho}(y):=\frac{1}{y}\int_{0}^{y}H_{\gamma,\rho}(x)dx - \int_{0}^{1}H_{\gamma,\rho}(x)dx.
\end{equation}

\subsection{Asymptotic Normality} \label{sec: Asymptotic Normality of smoothed estimator}
First, we introduce a useful approximation to the joint distribution of the CVaR order statistics of a sample. Let $\xi_i,\ i\geq1$, be independent standard exponential r.v. with partial sums
\begin{equation} \label{eqn: expo_sum}
    S_j = \sum_{i=1}^j \xi_i, \quad \text{for } j=1,2,\dots.
\end{equation}
Then, by Theorem 5.4.3, \cite{reiss2012approximate}, for any intermediate sequence $m_n$, we have the asymptotic equivalence (in distribution) that 
\begin{equation} \label{eqn: Reiss approx2}
    \Big\| \mathcal{L}\bigl((Y_j^{(n)})_{j=0}^m \bigr) - \mathcal{L}\left(\left(V(S_{j+1}/n)\right)_{j=0}^m \right) \Big\| = O(m/n), \quad n\to\infty,
\end{equation}
where $\mathcal{L}(\cdot)$ denotes the law of a random vector and $\|\cdot-\cdot\|$ is the variational distance of two distributions.

For $0<c<1$ and $\lambda\in\Lambda$, let
\begin{equation}
    \tilde{\gamma}_{n,m}^G(c,\lambda)=\int\log\{V(S_{\lfloor cj \rfloor+1} / n) - V(S_{j+1} / n)\} \lambda(dt),
\end{equation}
with $j=\lceil tm \rceil$. For convenience, abbreviate `with probability 1' to `wp1'. For $0<c<1$, we define the difference operator $\Delta_c$ on functions $f:(0,1]\rightarrow\mathbb{R}$ as follows,
\begin{equation} \label{eqn:delta_c}
    \Delta_c f(t):=\frac{1}{c}\int_0^c f(ut) du - \int_0^1 f(ut) du, \quad \text{for} \ 0<t\leq1.
\end{equation}
\begin{theorem} \label{thm: smooth_tilde_convergence}
Assume Conditions \ref{eqn: 1st_order_ERV_U} and \ref{cond2} hold with $\gamma<1/2$. Let $\Lambda^{\prime} \subset \Lambda$ be such that for some $0<\epsilon<1/2$ the integral $\int t^{1/2-\epsilon}|\lambda|(dt)$ is uniformly bounded over $\Lambda^{\prime}$. Let $0<c_0\leq c_n<1$ be such that $m_n^\eta (1-c_n)\rightarrow\infty$ for some $0<\eta<\epsilon/(2+4\epsilon)$. On a suitable probability space, there exists r.v. $\{S_j,j\geq1\}$ as in (\ref{eqn: expo_sum}) and a standard Wiener process $W$ such that wp1
\[
    \sup_{c0\leq c\leq c_n, \ \lambda\in\Lambda^{\prime}} |k_n^{1/2}\{\tilde{\gamma}_{n,m_n}^G(c,\lambda)-\gamma\} - Z_n(c,\gamma,\lambda) - rB(c,\gamma,\rho,\lambda)|\rightarrow 0,
\]
as $n\to\infty$, where
\[
    Z_n(c,\gamma,\lambda)=\int \frac{t^\gamma \Delta_c \{t^{-\gamma-1}W_n(t)\}}{\tilde{h}_\gamma(c)} \lambda(dt),
\]
\[
    B(c,\gamma,\rho,\lambda)=\int \frac{t^\gamma \Delta_c \{H_{\gamma,\rho}(t)\}}{\tilde{h}_\gamma(c)} \lambda(dt).
\]
and $W_n(t)=-m_n^{-1/2}W(m_n t)$ is a standard Wiener process as well.
\end{theorem}
Observe that for a standard Wiener process $W$, with $0<c<1,\gamma<1/2$, and a measure $\lambda\in\Lambda$ such that $\int t^{-1/2-\epsilon} < \infty$ for some $\epsilon>0$, the r.v.
\[
    Z(c,\gamma,\lambda)=\int \frac{t^\gamma \Delta_c \{t^{-\gamma-1}W(t)\}}{\tilde{h}_\gamma(c)} \lambda(dt)
\]
is mean-zero normally distributed with variance
\begin{equation} \label{eqn: asymptotic variance}
    v(c,\gamma,\lambda)=\int\int \sigma_{c,\gamma}(s,t)\lambda(ds)\lambda(dt),
\end{equation}
where
\begin{equation} \label{eqn:sigma}
    \sigma_{c,\gamma}(s,t)=E\Bigl[\frac{s^\gamma \Delta_c \{s^{-\gamma-1}W(s)\}}{\tilde{h}_\gamma(c)}\frac{t^\gamma \Delta_c \{t^{-\gamma-1}W(t)\}}{\tilde{h}_\gamma(c)}\Bigr].
\end{equation}
If we define $B_s(t)=\frac{1}{t}\int_0^t u^{-\gamma-1}W(us)du$, then
\[
    \sigma_{c,\gamma}(s,t)=\frac{1}{st\tilde{h}_\gamma^2(c)}\left[E\left(B_s(c)B_t(c)\right) -E\left(B_s(1)B_t(c)\right)-E\left(B_s(c)B_t(1)\right)+E\left(B_s(1)B_t(1)\right) \right],
\]
where
\[
    E\left(B_s(t_1)B_t(t_2)\right)=\begin{cases}
        (st)^\gamma(t_1t_2)^{-1}\left[h_1(\gamma)(st_1\land tt_2)^{1-2\gamma} - h_2(\gamma)(st_1\land tt_2)^{1-\gamma}(st_1\lor tt_2)^{-\gamma} \right], \quad \gamma\neq0, \\ \\
        (t_1t_2)^{-1} (st_1\land tt_2) \left[2 - \log{(st_1\land tt_2)} + \log{(st_1\lor tt_2)} \right], \quad \gamma=0,
    \end{cases}
\]
$h_1(\gamma)=\frac{1}{\gamma(1-\gamma)(1-2\gamma)}$, and $h_2(\gamma)=\frac{1}{\gamma(1-\gamma)}$.
Theorem \ref{thm: smooth_tilde_convergence} and approximation \eqref{eqn: Reiss approx2} together imply the asymptotic normality of the CVaR-based smoothed estimators.
\begin{corollary} \label{thm: asymptotic normality of smoothed estimator}
Assume Conditions \ref{cond1} and \ref{cond2} hold with $\gamma<1/2$. If $0<c<1$ and if $\lambda\in\Lambda$ is such that $\int t^{-1/2-\epsilon}|\lambda|(dt)<\infty$ for some $\epsilon$, then
\[
    \sqrt{m}\{\widehat{\gamma}_{n,m}(c,\lambda) - \gamma\} \Rightarrow \mathcal{N}\left(rB(c,\gamma,\rho,\lambda),v(c,\gamma,\lambda)\right).
\]
\end{corollary}

\section{The Optimal Measure Selection} \label{sec: measure selection}
The asymptotic variance of the estimator, as formulated by \cite{segers2005generalized}, is represented as $\int\int \tilde{\sigma}_{c,\gamma}(s,t)\lambda(ds)\lambda(dt)$, where $\lambda \in \Lambda$ and 
\[
    \tilde{\sigma}_{c,\gamma}(s,t)=\frac{1}{st c^{\gamma+1}h_\gamma^{2}(c)} \left[(c^{-\gamma}+c^{\gamma+1})(s \land t) - s \land (ct) - (cs) \land t\right].
\] 
Segers demonstrated that if a measure $\lambda \in \Lambda$ satisfies $\int \tilde{\sigma}_{c,\gamma}(s,t)\lambda(ds) = \beta_0 + \beta_1 \log{(1/t)}, \ 0<t\leq1$, with some $\beta_0\in\mathbb{R}$ and $\beta_1>0$, then $\beta_1$ is both the asymptotic variance and its optimal value. However, attempts to minimize the asymptotic variance of our estimator reveal that there is no solution for $\lambda\in\Lambda$ such that $\int \sigma_{c,\gamma}(s,t)\lambda(ds) = \beta_0 + \beta_1 \log{(1/t)}, \ 0<t\leq1$, with some $\beta_0\in\mathbb{R}$ and $\beta_1>0$. Thus, we restrict $\lambda\in\Lambda$ to specific classes of functions that satisfy \eqref{eqn: space}. We refer to one class of functions as the beta distribution's density function, offering considerable flexibility in function shape within the interval $(0,1]$. Other options include density functions of the Kumaraswamy distribution, logit-normal distribution, and several others. We opt for the beta distribution due to its convenience; modifying its density function analytically can perfectly satisfy \eqref{eqn: space}. Moreover, no other density functions outperform the beta density in the simulated scenarios. Thus, we propose the class of functions as our beta measure with two shape parameters $\alpha$ and $\beta$,
\begin{equation} \label{eqn: beta class}
    \lambda_{\alpha,\beta}(t) = \frac{t^{\alpha-1}(1-t)^{\beta-1}}{\text{Beta}(\alpha-1, \beta)}, 0<t\leq1,
\end{equation}
where $\text{Beta}(\alpha-1, \beta)=\frac{\Gamma(\alpha-1)\Gamma(\beta)}{\Gamma(\alpha+\beta-1)}$ and $\Gamma$ is the Gamma function. If not mentioned explicitly, the $\lambda$ that shows up later is always understood as $\lambda_{\alpha,\beta}$.

The optimal $\alpha$ and $\beta$ will be obtained by minimizing the asymptotic variance in \eqref{eqn: asymptotic variance}, and apparently, they vary across various $\gamma$ values. Since the expression for $v(c,\gamma,\lambda)$ is irregular, explicit formulas for the optimal values of $\alpha(\gamma)$ and $\beta(\gamma)$ are not attainable. However, they can still be computed numerically. One can also seek the optimal $\alpha(\gamma, \rho)$ and $\beta(\gamma, \rho)$ to minimize the value of $B(c,\gamma,\rho,\lambda)$, with the optimal measure $\lambda_{\text{bias}}^*$, which is directly proportional to the asymptotic bias.

Recall that $r$ is defined as $\lim_{n\to\infty} d \sqrt{m} A(\frac{n}{m})$ in Condition \ref{cond2}, and we let $r_{n,m}:=\sqrt{m}A(\frac{n}{m})$. According to the Corollary \ref{thm: asymptotic normality of smoothed estimator}, the variance of the estimator $\widehat{\gamma}_{n,m}(c,\lambda)$ is significant for small $m$ (where $A(n/m)=0$ and $r_{n,m}=0$). However, increasing $m$ (such that $A(n/m),r_{n,m}\in(0,\infty)$) to reduce the variance results in the estimator having a non-zero bias. To address this, we first introduce the asymptotic mean squared error (AMSE) of the estimator, defined as
\begin{equation} \label{eqn: AMSE}
    \text{AMSE}(\widehat{\gamma}_{n,m}(c,\lambda)) = \frac{1}{m}\left[r_{n,m}^2 B^2(c,\gamma,\rho,\lambda) + v(c,\gamma,\lambda)\right].
\end{equation}
The value of $r_{n,m}^2$ is affected by the selection of the sequence $m=m(n)$, and $A(n/m)=r_{n,m}/\sqrt{m}$. One might want to minimize the AMSE with respect to $\lambda$ from \eqref{eqn: AMSE} and seek the optimal $\alpha(\gamma, \rho)$ and $\beta(\gamma, \rho)$. Unfortunately, the value of the AMSE depends on the unknown second-order scale function $A(n/m)$ (and so the function $r_{n,m}$). This fact makes obtaining a practical strategy for minimizing the AMSE difficult. In the existing literature, the estimation of the second-order scale function always serves as a byproduct of the estimation of $\rho$ for the second-order reduced-bias (SORB) estimator. We mention the pioneering papers by \cite{peng1998asymptotically}, \cite{beirlant1999tail}, \cite{feuerverger1999estimating} and \cite{ivette2000alternatives}, among others, based on log-excesses and on scaled log-spacings between subsequent extreme order statistics from a Pareto-type distribution. In these papers, SORB estimators are proposed with asymptotic variances larger than $\gamma^2$ which is the asymptotic variance of the Hill estimator. \cite{caeiro2005direct, caeiro2009reduced}, \cite{gomes2007improving}, and \cite{ivette2008tail} reduce the bias while keeping asymptotic variance the same as that of the Hill estimator. These estimators are called minimum-variance reduced-bias (MVRB) estimators with the specific case $A(t)=\gamma C t^\rho$ and an adequate `external' consistent estimation of the pair of second-order parameters, $(C,\rho) \in (\mathbb{R}, \mathbb{R}^-)$. In the general MDA, $F\in\mathscr{D}(G_\gamma)$, but not just Pareto-type distribution, \cite{cai2013bias} introduces a SORB estimator for $\gamma$ around zero based on the probability weighted moment (PWM) methodology and the estimator for $A$ is also provided. In this paper, we propose an algorithm to smoothly estimate $A$ in a nonparametric way across varying $m$. We first define the approximated AMSE as
\begin{equation} \label{eqn: approx AMSE}
    \widetilde{\text{AMSE}}(\widehat{\gamma}_{n,m}(c,\lambda)) = \frac{1}{m}\left[\widehat{r}_{n,m}^2 B^2(c,\gamma,\rho,\lambda) + v(c,\gamma,\lambda)\right],
\end{equation}
and $\widehat{A}(n/m)=\widehat{r}_{n,m}/\sqrt{m}$, where $\widehat{r}_{n,m}$ is the estimated $r_{n,m}$. 
We denote the optimal AMSE measure as $\lambda_{\text{AMSE}}^*$ which minimizes the approximated AMSE in \eqref{eqn: approx AMSE}. The following heuristic algorithm \ref{algo: r_{n,m}} to approximate the second-order scale function $A$ (and the function $r_{n,m}$) via the bootstrapping approach is suggested. When taking a bootstrap sample $\left\{X_i^*\right\}_{i=1}^n$, we order the bootstrapped observations as $X_1^{*(n)} \geqslant X_2^{*(n)} \geqslant \cdots \geqslant X_n^{*(n)}$. Our estimator based on the bootstrapped sample is then denoted as $\widehat{\gamma}_{n,m}^*(c,\lambda)$. \cite{de2024bootstrapping} prove the bootstrap analogue of the tail quantile
process of the Peaks-over-Threshold (POT) method and show by simulations that the sample variance of bootstrapped estimates can be a good approximation for the asymptotic variance of the PWM estimator as an example. They point out that any other estimators using the POT method, as long as its asymptotic behavior can be
developed as a linear functional of the tail quantile process, can still hold such approximation. We assume this bootstrap procedure can also provide a good approximation of asymptotic bias. Note that in the algorithm, we take our estimator with $\lambda_{\text{bias}}^*$ as the `unbiased' estimator to compute MSE, and in Section \ref{sec: simulations} the estimator with the measure $\lambda_{\text{bias}}^*$ when taking $m=n$ is shown to be the most `consistent' one by simulations. 

\begin{algorithm}[H]
\SetAlgoLined
\KwIn{$X$: Input data, $m_0$: Size of upper order statistics from data chosen}
\SetKwInOut{Parameter}{Parameters}
\Parameter{$B=1000$: Bootstrapping size, $r_{\text{max}}=15$: Largest value of the grid for $r$, $\Delta m$=4: Step size across $m$, $d=0.5$: Distance threshold, $\tau=0.1$: Step threshold}
\KwOut{$r_{n,m_0}$: Estimated value of $\sqrt{m_0}A(\frac{n}{m_0})$}
\SetKwFunction{FindValueR}{FindValueR}
\SetKwProg{Fn}{Function}{:}{}
\Fn{\FindValueR{$X$, $m_0$, $B$, $r_{\text{max}}$, $\Delta m$, $d$, $\tau$}}{
    Generate a grid of $r$ values $\left\{r_j\right\}$ from $0$ to $r_{\text{max}}$\;
    Generate $B$ bootstrapped samples from X\;
    \For{$m \gets \Delta m$ \textbf{to} $m_0$ \textbf{by} $\Delta m$}{
        \eIf{$m=\Delta m$}{
            $r_{n,m} \gets 0$\;
        }{
            Compute each analytical $\text{AMSE}_j := \text{AMSE}(\widehat{\gamma}_{n,m}(c,\lambda_{\text{AMSE}_j}^*))$ conditioned on $\left\{r_j\right\}$ in \eqref{eqn: AMSE}\;
            Compute each estimated MSE from bootstrapped samples with the optimal minimal AMSE measure, $\widehat{\text{MSE}}_j := \widehat{\text{MSE}}(\widehat{\gamma}_{n,m}^*(c,\lambda_{\text{AMSE}_j}^*))$, conditioned on $\left\{r_j\right\}$\;
            \eIf{$m\cdot\text{AMSE}_j$ and $\widehat{\text{MSE}}_j$ have intersection points across $r_j$ (choose the first intersection point $r^*$), and $|r^*-r_{n,m-\Delta m}| \leq d$}{
                $r_{n,m} \gets r^*$\;
            }{
                Find $r_{n,m}$ to minimize $|\widehat{\text{MSE}}_j - m\cdot\text{AMSE}_j|$ with $r\in[\max\{r_{n,m-\Delta m}-\tau, r_{\text{left}}\}, \min\{r_{n,m-\Delta m}+\tau, r_{\text{right}}\}]$ from $\left\{r_j\right\}$ where $r_{\text{left}}$ and $r_{\text{right}}$ are local maximums of $|\widehat{\text{MSE}}_j - m\cdot\text{AMSE}_j|$ around $r_{n,m-\Delta m}$\;
            }
        }
    }
    \KwRet $r_{n,m_0}$\;
}
$r_{n,m_0} \gets$ \FindValueR{$X$, $m_0$, $B$, $r_{\text{max}}$, $\Delta m$, $d$, $\tau$}\;
\caption{Find the value of $r_{n,m_0}$}
\label{algo: r_{n,m}}
\end{algorithm}

Alternatively, we introduce a regularized mean squared error (RegMSE) defined as
\begin{equation} \label{eqn: RegMSE}
    \text{RegMSE}(\widehat{\gamma}_{n,m}(c,\lambda)) = \frac{1}{m}\left[r_C^2 B^2(c,\gamma,\rho,\lambda) + v(c,\gamma,\lambda)\right],
\end{equation}
where $r_C$ can be recognized as the regularization parameter for bias. Optimal values for \(\alpha\) and \(\beta\) can be numerically determined by minimizing the Regularized Mean Squared Error (RegMSE) of the estimator, where these parameters vary with both \(\gamma\) and \(\rho\). Heuristically, we experimented with different values of \(r_C\) and found that \(r_C = 15 \cdot \frac{m}{n}\) consistently yields the best performance across all distributions. In Section \ref{sec: simulations}, both methods that minimize the approximate Asymptotic Mean Square Error (AMSE) and the RegMSE are utilized in simulations.

Since the optimal measure of $\lambda_{\alpha,\beta}$ depends on the unknown $\gamma$ and the second-order parameter $\rho$ which is difficult to estimate, we can employ the adaptive estimator $\widehat{\gamma}_{n,m}(c,\lambda_{\alpha^*(\Bar{\gamma},\Bar{\rho}), \beta^*(\Bar{\gamma},\Bar{\rho})})$. Here, $\Bar{\gamma}$ represents the initial estimate of $\gamma$, and $\Bar{\rho}<0$ is an assumed value and may or may not be equal to the true $\rho$. Further details regarding the estimator $\Bar{\gamma}$ for the initial estimate will be presented in Section \ref{sec: simulations}.

\section{Simulations} \label{sec: simulations}
In the simulation study, we include three CVaR-based estimators: (a) \textit{CVaR-based Pickands Estimator}: $\widehat{\gamma}_{n, m}(2, 2)$, where the parameters are selected to be the same as those in \cite{pickands1975statistical}. (b) \textit{Adaptive CVaR-based Smoothed Estimator Minimizing RegMSE}: $\widehat{\gamma}_{n,m}^{\text{CVaR}}\left(c, \lambda_{\alpha^*(\Bar{\gamma},\Bar{\rho}), \beta^*(\Bar{\gamma},\Bar{\rho})}^{RegMSE}\right)$ with initial estimates $\Bar{\gamma}$ and $\Bar{\rho}$. This estimator uses the beta measure that minimizes the regularized mean squared error (RegMSE) as defined in equation \eqref{eqn: RegMSE}. (c) \textit{Adaptive CVaR-based Smoothed Estimator Minimizing Approximate AMSE}: $\widehat{\gamma}_{n,m}^{\text{CVaR}}\left(c, \lambda_{\alpha^*(\Bar{\gamma},\Bar{\rho}), \beta^*(\Bar{\gamma},\Bar{\rho})}^{AMSE}\right)$ with initial estimates $\Bar{\gamma}$ and $\Bar{\rho}$. This estimator employs the beta measure that minimizes the approximate asymptotic mean squared error (AMSE) as given in equation \eqref{eqn: approx AMSE}. The parameters $c$ and $\Bar{\rho}$ are set to match those in the simulation by \cite{segers2005generalized}, which are $0.75$ and $-1$, respectively. The initial estimator for cases (b) and (c) is chosen as $\Bar{\gamma} = \widehat{\gamma}_{n,m}^{\text{CVaR}}\left(c, \lambda_{\alpha^*(0), \beta^*(0)}^{Var}\right)$, using the beta measure that minimizes the asymptotic variance in equation \eqref{eqn: asymptotic variance} and assuming a value of $0$ for $\gamma$. For convenience, we denote the estimators in (b) and (c) as $\widehat{\gamma}_{n,m}^{\text{CVaR}}\left(c, \lambda_{\Bar{\gamma},\Bar{\rho}}^{Reg^*}\right)$ and $\widehat{\gamma}_{n,m}^{\text{CVaR}}\left(c, \lambda_{\Bar{\gamma},\Bar{\rho}}^{A^*}\right)$, respectively. We also enhance the adaptive generalized Pickands (VaR-based smoothed) estimator proposed by \cite{segers2005generalized} by using beta measures that minimize its regularized mean squared error (RegMSE), denoted as $\widehat{\gamma}_{n,m}^{\text{VaR}}\left(c, \lambda_{\Bar{\gamma}^V,\Bar{\rho}}^{Reg^*}\right)$. The initial estimates for this estimator are set to $\Bar{\gamma}^V = \widehat{\gamma}_{n,m}^{\text{VaR}}\left(c, \lambda_{\alpha^*(0), \beta^*(0)}^{Var}\right)$ and $\Bar{\rho} = -1$. Additionally, we compare the adaptive unconstrained and constrained generalized Pickands estimators by \cite{segers2005generalized}, denoted as $\widehat{\gamma}_{n,m}^S\left(c, \widetilde{\lambda}_{c,\Bar{\gamma}^S}\right)$ and $\widehat{\gamma}_{n,m}^S\left(c, \widetilde{\lambda}_{c,\Bar{\gamma}^S}\right)$. The measures $\widetilde{\lambda}_{c,\gamma}$ and $\widetilde{\lambda}_{c,\gamma,\rho}$ are derived analytically to minimize asymptotic variance, and to eliminate bias while minimizing asymptotic variance, respectively. The initial estimates for these optimal measures are chosen to be $\Bar{\gamma}^S = \widehat{\gamma}_{n,m}^S\left(c, \widetilde{\lambda}_{c,0}\right)$ and $\Bar{\rho} = -1$. We conduct simulations to demonstrate the finite-sample behaviors and performances of all the estimators listed in Table \ref{tab: estimators}. Unless stated otherwise, all estimators, except for the CVaR-based Pickands estimator, follow an adaptive procedure with an initial estimate of $\gamma$ (and $\rho$, if needed). Furthermore, we categorize the last three estimators in Table \ref{tab: estimators} as the group of AMSE-minimized estimators.

\begin{table}[H] 
    \centering
    \renewcommand{\arraystretch}{1.5} 
    \begin{tabular}{l|l}
        \hline
        \multicolumn{2}{c}{Estimators} \\
        \hline
        $\widehat{\gamma}_{n, m}(2, 2)$ & CVaR-based Pickands \\ $\widehat{\gamma}_{n,m}^S(c,\widetilde{\lambda}_{c,\Bar{\gamma}^S})$ & Generalized Pickands unconstrained with $c=0.75$ \\ $\widehat{\gamma}_{n,m}^S(c,\widetilde{\lambda}_{c,\Bar{\gamma}^S,\Bar{\rho}})$ & Generalized Pickands constrained with $c=0.75$ \\ $\widehat{\gamma}_{n,m}^{\text{VaR}}(c, \lambda_{\Bar{\gamma}^V,\Bar{\rho}}^{Reg^*})$ & VaR-based smoothed with $c=0.75$ and RegMSE minimized \\ $\widehat{\gamma}_{n,m}^{\text{CVaR}}(c, \lambda_{\Bar{\gamma},\Bar{\rho}}^{Reg^*})$ & CVaR-based smoothed with $c=0.75$ and RegMSE minimized \\
        $\widehat{\gamma}_{n,m}^{\text{CVaR}}(c, \lambda_{\Bar{\gamma},\Bar{\rho}}^{A^*})$ & CVaR-based smoothed with $c=0.75$ and approximate AMSE minimized \\
        \hline
    \end{tabular}
    \caption{Estimators employed in simulations.}
    \label{tab: estimators}
\end{table}

Table \ref{tab: dist} summarizes distributions involved in the experiments. They are chosen to cover a wide range of cases for $(\gamma, \rho)$. From each distribution, 1000 samples of size $n=1000$ is generated. The $k$-th upper order statistic represents the smallest upper order statistic used in the estimation, i.e., $k$ denotes the size of data employed in the estimation. The CVaR-based Pickands estimator, $\widehat{\gamma}_{n, m}(2, 2)$, is computed at $m=k/4=2, 3, \ldots, 250$ while other estimators in Table \ref{tab: estimators} are computed at $m=k=8, 12, \ldots, 1000$. 

\begin{table}[H] 
% \footnotesize
\centering
\begin{tabular}{lcc} 
\hline
\specialrule{0em}{2pt}{2pt}
% \specialrule{0.05em}{3pt}{3pt}
Distribution & $\gamma$ & $\rho$ \\
\specialrule{0em}{2pt}{2pt}
\hline \hline
$\gamma>0$ & & \\
\tabitem GPD $F(x) = 1-\left(1+\gamma x \right)^{-\frac{1}{\gamma}}$, $x \geqslant 0$, with $\gamma = 0.25$ & $\gamma$ & $-\infty$ \\
\tabitem Fréchet $F(x) = \exp \left(-x^{-\alpha}\right)$, $x>0$, with $\alpha=4$ & $1/\alpha$ & $-1$ \\
\tabitem Burr $F(x)=1-\left(1 + x^{c}\right)^{-k}, x \geqslant 0$, with $c=2$ and $k=2$ & $1/ck$ & $-1/k$ \\
\tabitem Student-$t$ with $v=4$ degrees of freedom & $1 / v$ & $-2 / v$ \\
\\
$\gamma=0$ & & \\
\tabitem Exponential $F(x)=1-\mathrm{e}^{-x}, x \geqslant 0$ & 0 & $-\infty$ \\
\tabitem Gumbel $F(x)=\exp \left(-\mathrm{e}^{-x}\right), x \in \mathbb{R}$ & 0 & $-1$ \\
\tabitem Logistic $F(x)=1-2 /\left(1+\mathrm{e}^{x}\right), x \geqslant 0$ & 0 & $-1$ \\
\tabitem Weibull $F(x)=1-\exp \left(- x^{\tau}\right), x \geqslant 0$, with $\tau=1.5$ & 0 & $0$ \\
\tabitem Lognormal  with $\sigma = 1$ & $0$ & $0$ \\
\tabitem Standard normal & $0$ & $0$ \\
\\
$\gamma<0$ & & \\
\tabitem GPD $F(x) = 1-\left(1+\gamma x \right)^{-\frac{1}{\gamma}}$, $0\leqslant x \leqslant -\frac{1}{\gamma}$, with $\gamma= -0.2$ & $\gamma$ & $-\infty$ \\
\tabitem Reversed Weibull $F(x) = \exp \left(-|x|^\alpha\right)$, $x\leqslant0$, with $\alpha=5$ & $-1/\alpha$ & $-1$ \\
\tabitem Uniform & -1 & $-\infty$ \\
\hline
\end{tabular}
\caption{Distributions Employed in Simulation Study}
\label{tab: dist}
\end{table}

To compare the efficiency of the estimators, we adopt the commonly used mean squared error (MSE) which combines the bias and variance at the same time. In Figures \ref{fig: burr0.25}, \ref{fig: frechet0.25}, and \ref{fig: gpd0.25} below, the MSE results are drawn across varying $k(m)$. Hence, they are important not only for the evaluation of the efficiency of the estimators; but also to show the stability of the results across the choice of $m$. In addition to the MSE results, we also provide the estimates of $\gamma$ as a function of $k(m)$. They illustrate the smoothness of the estimator concerning changing $m$. 

% Burr 0.25
\begin{figure}[H]
    \begin{subfigure}{0.5\textwidth}
        \includegraphics[width=\linewidth]{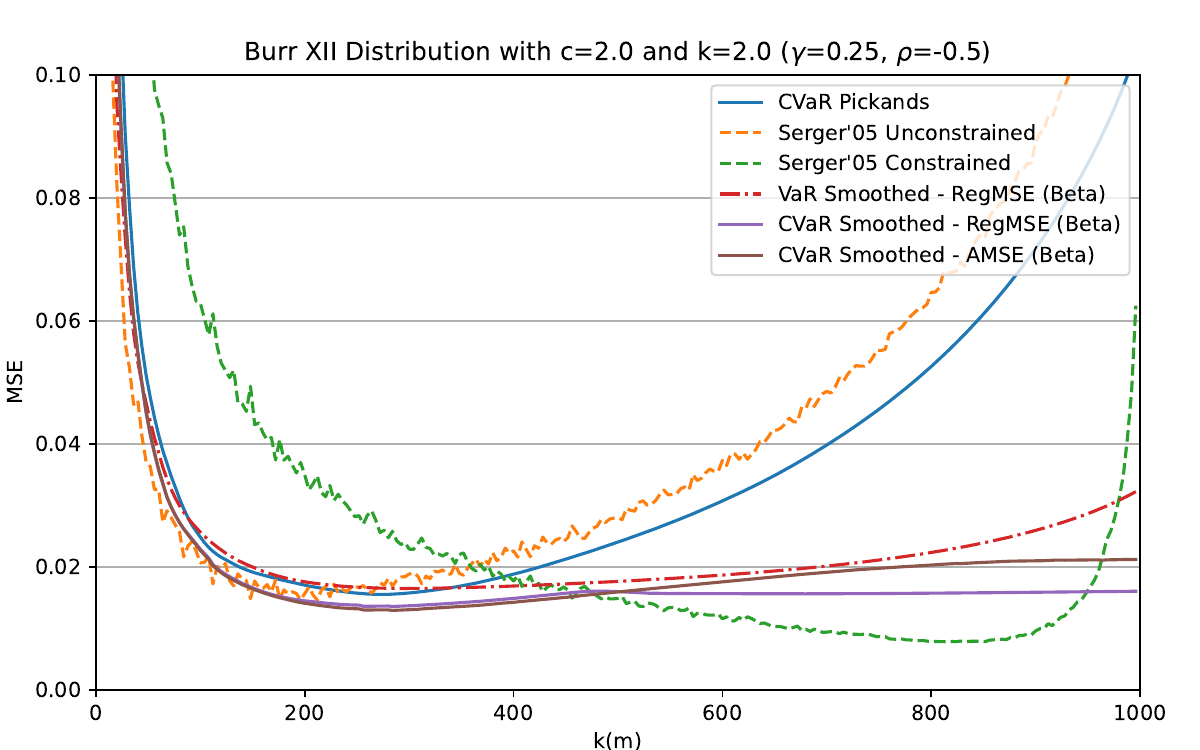}
        \caption{MSE}
        \label{fig: burr0.25_(a)}
    \end{subfigure}\hfill
    \begin{subfigure}{0.5\textwidth}
        \includegraphics[width=\linewidth]{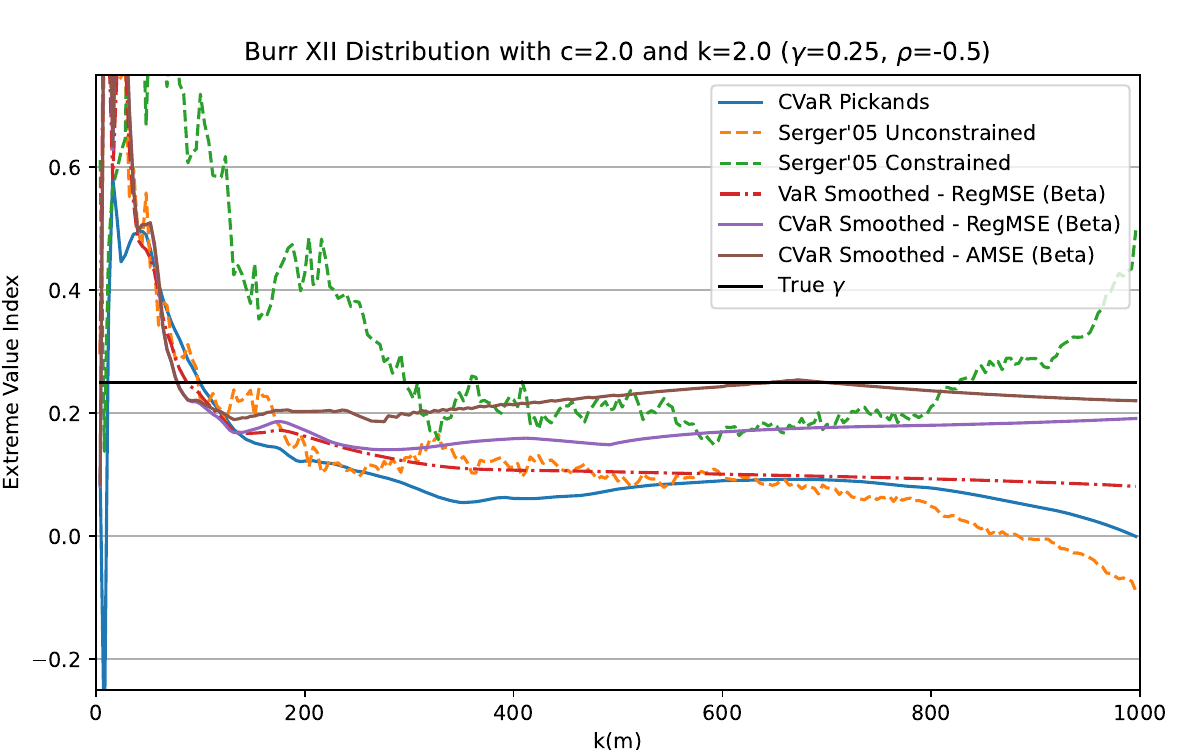}
        \caption{Estimates of $\gamma$}
        \label{fig: burr0.25_(b)}
    \end{subfigure}
    \caption{MSE (1000 samples) and the estimates of $\gamma$ from the first sample against $k(m)$ smallest upper order statistics. The samples are from the the Burr distribution with $\gamma=0.25$}
    \label{fig: burr0.25}
\end{figure}

% GEV 0.25
\begin{figure}[H]
    \begin{subfigure}{0.5\textwidth}
        \includegraphics[width=\linewidth]{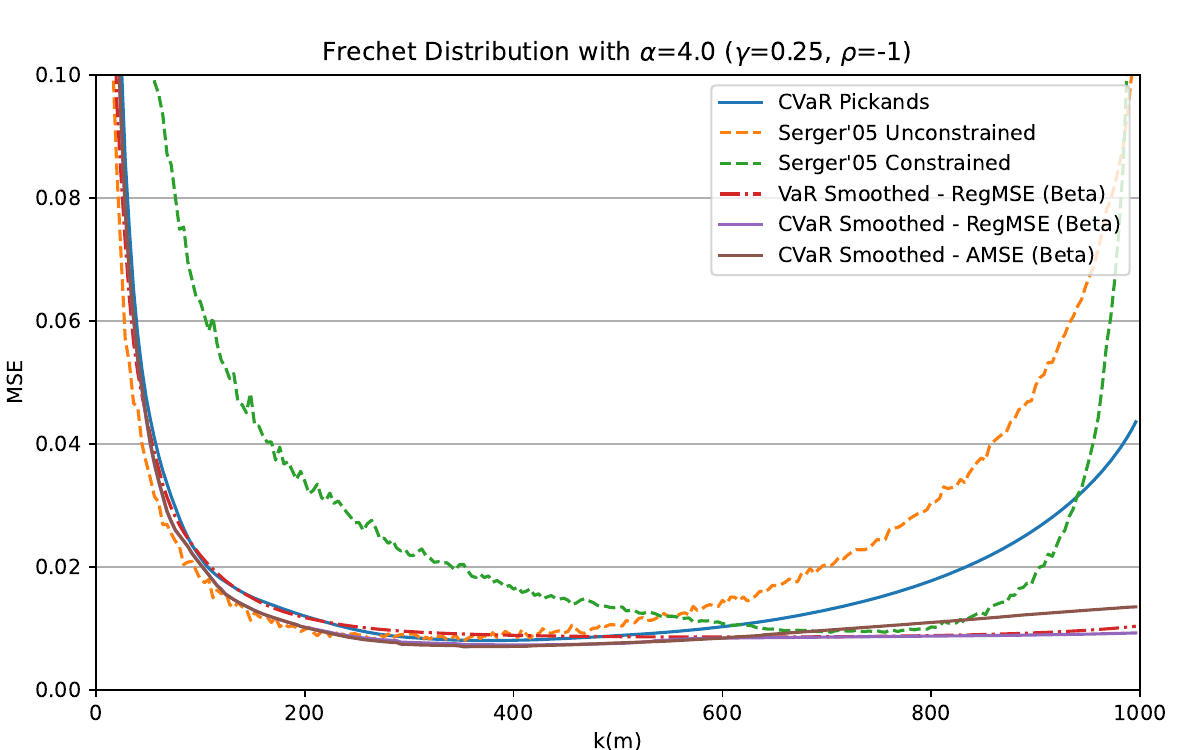}
        \caption{MSE}
        \label{fig: frechet0.25_(a)}
    \end{subfigure}\hfill
    \begin{subfigure}{0.5\textwidth}
        \includegraphics[width=\linewidth]{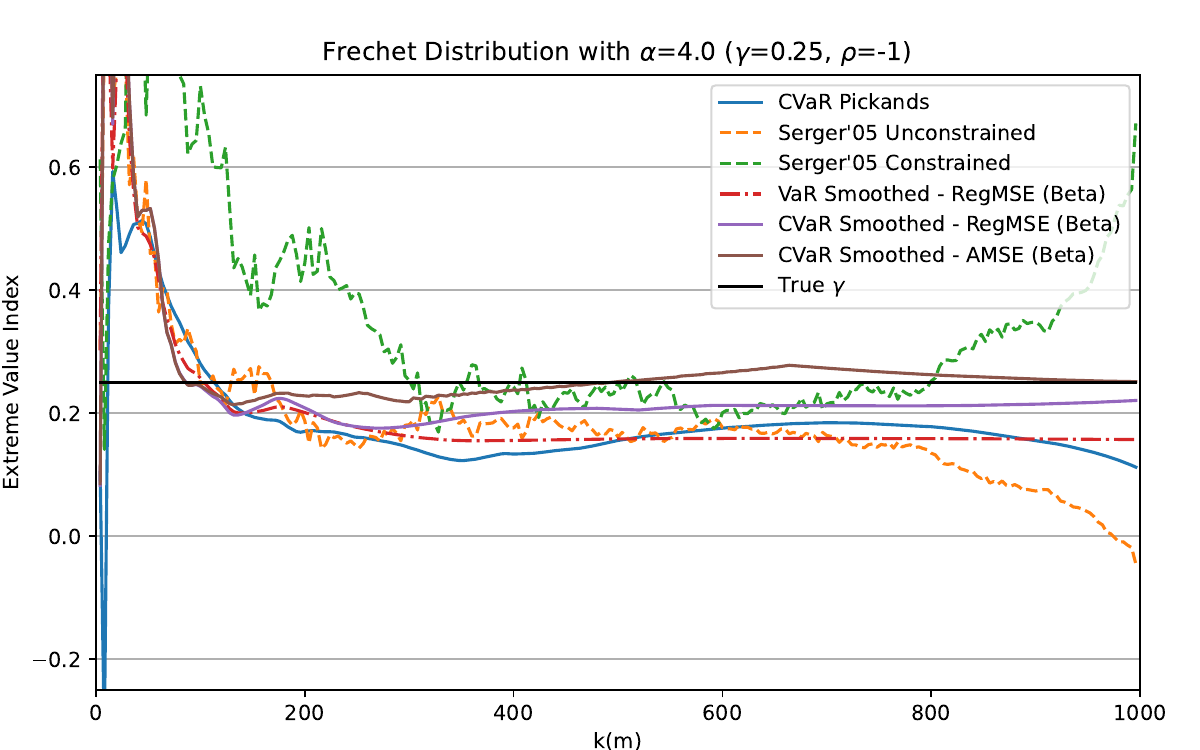}
        \caption{Estimates of $\gamma$}
        \label{fig: frechet0.25_(b)}
    \end{subfigure}
    \caption{The samples are from the Fréchet distribution with $\gamma=0.25$, and others are the same as Figure \ref{fig: burr0.25}}
    \label{fig: frechet0.25}
\end{figure}

% GPD 0.25
\begin{figure}[H]
    \begin{subfigure}{0.5\textwidth}
        \includegraphics[width=\linewidth]{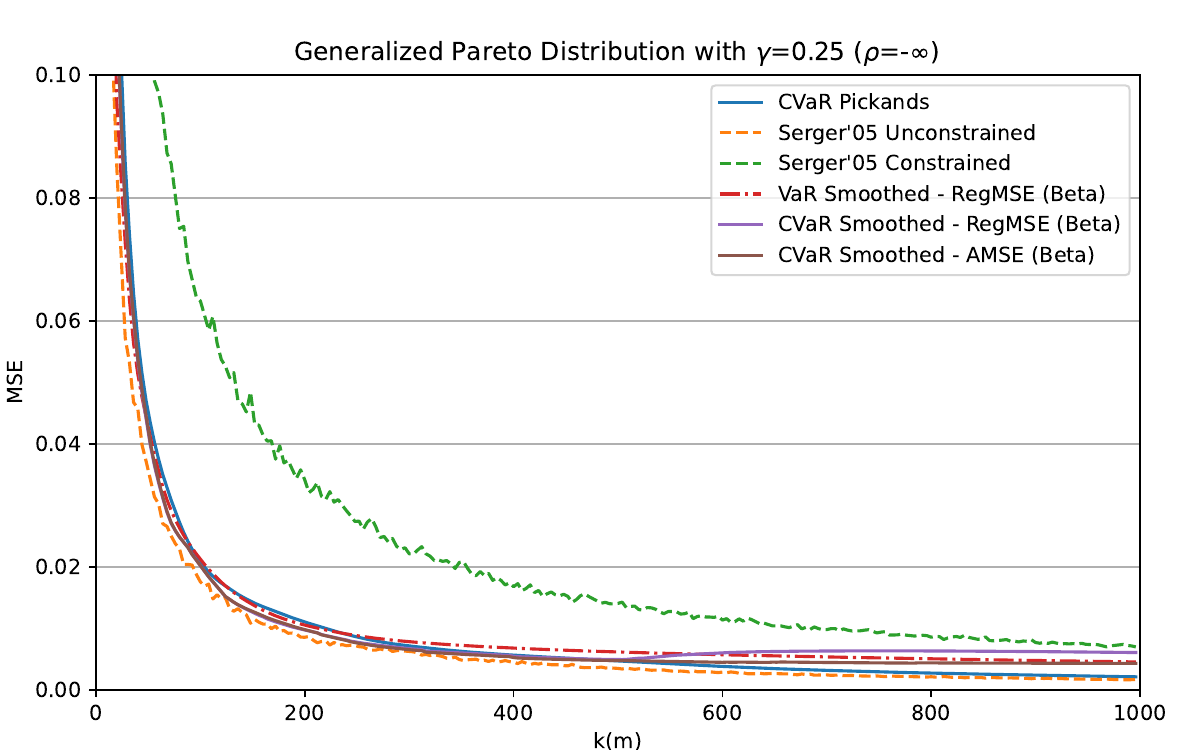}
        \caption{MSE}
        \label{fig: gpd0.25_(a)}
    \end{subfigure}\hfill
    \begin{subfigure}{0.5\textwidth}
        \includegraphics[width=\linewidth]{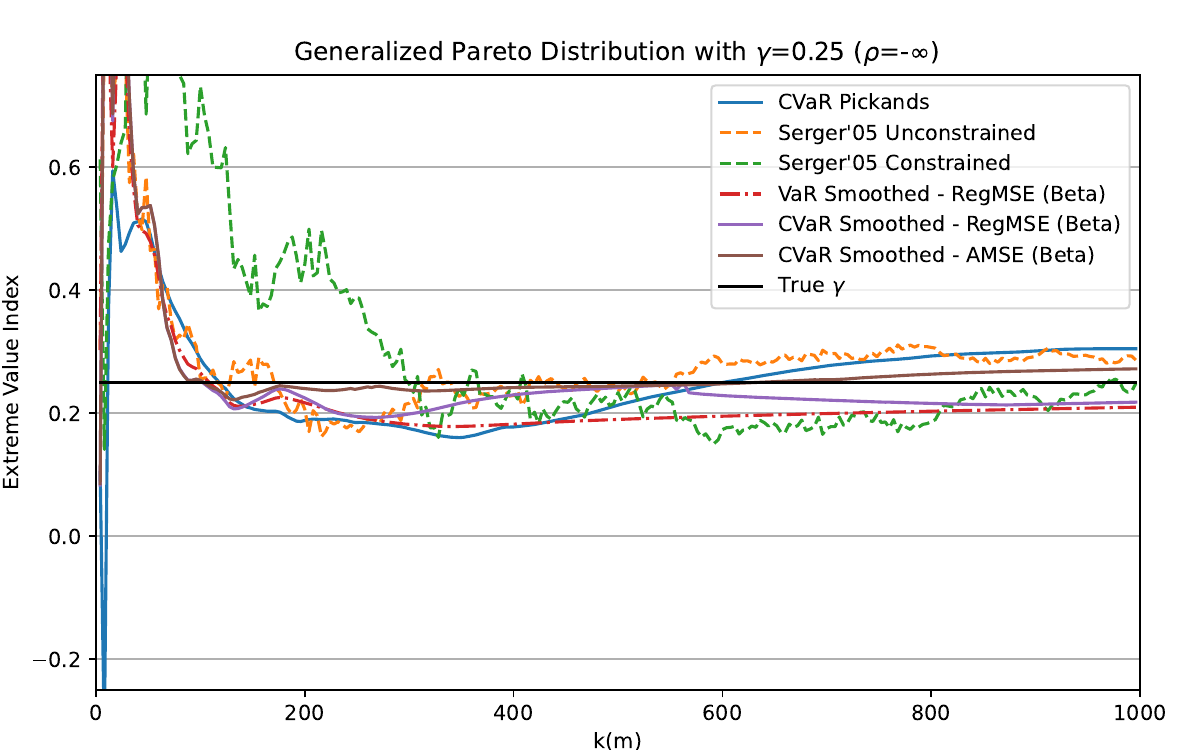}
        \caption{Estimates of $\gamma$}
        \label{fig: gpd0.25_(b)}
    \end{subfigure}
    \caption{The samples are from the the Generalized Pareto distribution with $\gamma=0.25$, and others are the same as Figure \ref{fig: burr0.25}}
    \label{fig: gpd0.25}
\end{figure}

The estimation results for all estimators in Table \ref{tab: estimators} from GPD, Fréchet, and Burr, with $\gamma=0.25$ are depicted in Figures \ref{fig: burr0.25}, \ref{fig: frechet0.25}, and \ref{fig: gpd0.25}, respectively. Additional results for other distributions from Table \ref{tab: dist} are provided in the Appendix \ref{appendix: simulations}. For all distributions with various $\gamma$, two distinct types of behavior of estimators can be observed as the choice of $k$ varies. For distributions with $\rho$ distant from zero (indicating rapid convergence of the excess distribution to the GPD), such as GPD in Figures \ref{fig: gpd0.25}, exponential, and uniform, we observe that as $k$ increases, except for CVaR-based smoothed estimator with RegMSE minimized, estimates become more accurate and converge towards the true extreme value index $\gamma$. While for distributions with $\rho$ close to zero, such as Burr in Figure \ref{fig: burr0.25}, Fréchet in Figure \ref{fig: frechet0.25}, it is observed that initially increasing $k$ leads to improved accuracy in estimation. However, as $k$ continues to increase and approaches its maximum value, which is the sample size $n=1000$, all estimators, except for constrained and smoothed AMSE-minimized estimators, tend to underestimate $\gamma$ (see Figures \ref{fig: burr0.25_(a)} and \ref{fig: frechet0.25_(a)}). This leads to a negative bias and, consequently, a sharp increase in MSE. In contrast, the group of AMSE-minimized estimators effectively controls the MSE across different values of $m$. Within this group, we also observe that their performances vary as $\rho$ changes. Hence, our primary focus is on the results of Burr, Fréchet, and GPD with $\gamma=0.25$ in Figures \ref{fig: burr0.25}, \ref{fig: frechet0.25}, and \ref{fig: gpd0.25}, respectively, as three representative examples. The results for the other distributions such as Student-$t$, Gumbel, Weibull, logistic, lognormal, and reversed Weibull exhibit similar patterns, see Appendix \ref{appendix: simulations} for more details.

For Fréchet and Burr distributions with $\gamma = 0.25$ (where $\rho$ is close to zero), as shown in Figures \ref{fig: burr0.25_(a)} and \ref{fig: frechet0.25_(a)}, we observe that for all estimators not in the group of AMSE-minimized estimators, the performance in terms of MSE is not stable across different values of $k$. As $k$ increases, the MSE for these estimators starts to increase sharply. In Figures \ref{fig: burr0.25_(b)} and \ref{fig: frechet0.25_(b)}, these estimators (except for the constrained generalized Pickands estimator) exhibit a noticeable downward trend in their estimates, leading to a rapid increase in MSE. However, the MSE behavior of the estimators in the AMSE-minimized group remains remarkably stable as a function of $k(m)$. In Figures \ref{fig: burr0.25_(b)} and \ref{fig: frechet0.25_(b)}, we can see that the downward trend in their estimates is significantly suppressed, with estimates varying around the true $\gamma$. When only comparing the constrained and unconstrained generalized Pickands estimators with the VaR-based smoothed estimator with minimized RegMSE, the effect of AMSE minimization is evident. Within the group of AMSE-minimized estimators, the CVaR-based smoothed estimator with minimized RegMSE demonstrates the most stable performance and the greatest smoothness in estimation concerning changes in $k$ (and consequently $m$), especially as $k$ approaches $n$, which involves incorporating all data into the estimation. This estimator also consistently shows the lowest MSE across nearly all values of $k$, except for a short range near $n$ for the Burr distribution, where it is slightly outperformed by the constrained generalized Pickands estimator. Therefore, the CVaR-based smoothed estimator with minimized RegMSE performs the best compared with others. 
For GPD with $\gamma = 0.25$ (where $\rho$ is far from zero), the unconstrained generalized Pickands estimators demonstrate superior performance in terms of MSE compared to the other estimators, as illustrated in Figure \ref{fig: gpd0.25_(a)}. The AMSE-minimized estimators perform slightly worse due to the large difference between the assumed $\bar{\rho} = -1$ and the true $\rho = -\infty$. However, we can observe that their performances are still comparable.
Additionally, from these distribution results, we can observe that the CVaR-based smoothed estimators are less sensitive to $\rho$ (and thus the discrepancy between $\bar{\rho}$ and $\rho$) compared to the VaR-based smoothed estimator. For distributions where $\rho$ is far from zero, although the CVaR-based smoothed estimators do not outperform the VaR-based smoothed estimator with minimized RegMSE, their performance difference is not substantially large, and the CVaR-based smoothed estimators still maintain a stable performance with a low MSE. In contrast, for distributions where $\rho$ is close to zero, particularly the Burr distribution, the CVaR-based smoothed estimators demonstrate notably more stability and lower MSE compared to the VaR-based smoothed estimator.

Based on the results from all distributions in Table \ref{tab: dist}, two estimators stand out: the unconstrained generalized Pickands estimator by \cite{segers2005generalized}, and the CVaR-based smoothed estimator with RegMSE minimized. Three criteria are used for assessing the superiority of an estimator: (a) low MSE; (b) smooth estimation, characterized by small changes as $m$ varies; (c) stable MSE within the range of $m$. Empirical analysis guides the determination of the best estimator, considering these criteria. The decisive factor in choosing the estimator is the second-order parameter $\rho$ of the distribution. In particular, for large negative $\rho$, the unconstrained generalized Pickands estimator emerges as the best choice; for $\rho$ close to $0$, the CVaR-based smoothed estimator with RegMSE minimized is preferable, offering optimal MSE across nearly the entire range of $m$ and demonstrating a more stable performance and smoother estimates as $m$ varies.

\section{Conclusion} \label{sec: conclusion}
This paper incorporates the CVaR order statistics quantity into the generalized class of Pickands estimators developed by \cite{segers2005generalized}. The weak consistency and asymptotic normality of the proposed CVaR-based smoothing estimator are also established. Furthermore, we propose a modified beta measure for smoothing and a heuristic algorithm to approximate AMSE. The objective can also be replaced by RegMSE which is constructed in a heuristic manner. The optimal beta measure is then computed numerically. Alongside the introduction of the CVaR order statistics as an additional smoothing technique, our estimator achieves a lower MSE, a smoother estimate, and a very stable performance across different intermediate order statistics. Based on the three criteria and the simulation results presented in Section \ref{sec: simulations}, our CVaR-based smoothed estimator with minimized RegMSE demonstrates notable advantages overall. The latter two criteria are particularly important in practice as they reduce the need for precise selection of $m$. 

In practical applications, we recommend using the adaptive CVaR-based smoothed estimator with minimized RegMSE, $\widehat{\gamma}_{n,m}^{\text{CVaR}}(c, \lambda_{\Bar{\gamma},\Bar{\rho}}^{Reg^*})$, as a universal estimator for the extreme value index $\gamma$ when $\gamma < 1/2$. This estimator demonstrates consistently smooth estimation and stable MSE performance across varying $m$. One might consider selecting $m$ within the range of 30\% to 100\% of the sample size $n$. When $\rho$ is close to zero, this estimator achieves the best MSE results across nearly all values of $m$. Even when $\rho$ is far from zero, the estimator achieves very comparable results with low MSE relative to the unconstrained generalized Pickands estimator. Therefore, we still opt to use the CVaR-based smoothed estimator with RegMSE minimized. 
In practice, the value of the second-order parameter $\rho$ is unknown, necessitating an initial estimation of it. Even when $\rho$ is known or accurately estimated, our estimator will benefit from this information via bias correction and have an improved performance.

\appendix
\section{Proof of Theorem \ref{thm: weak consistency}} \label{appendix: weak consistency}
Theorem \ref{thm: weak consistency} is an intermediate consequence of the CVaR-based Reiss's approximation \eqref{eqn: Reiss approx2} and the following property.
\begin{proposition}
Let $F\in\mathscr{D}(G_\gamma)$ with $\gamma<1$, $0<c<1$, and $\lambda\in\Lambda$. For every intermediate $m$ and every sequence $\{s_j\}_{j\geq1}$ of positive numbers such that $\lim_{j\to\infty}s_j/j=1$, we have
\begin{equation} \label{eqn: intermediate consistency}
    \lim_{n\to\infty}\int_{(0,1]} \log{\left\{V(n/s_{\lfloor c\lceil tm_n\rceil\rfloor+1}) - V(n/s_{\lceil tm_n\rceil+1})\right\}} \lambda(dt) = \gamma.
\end{equation}
\end{proposition}
\begin{proof}
By assumption $\lambda(0,1]=0$ and $\int\log{(1/t)}\lambda(dt)=1$. Abbreviate $j=\lceil tm_n\rceil$ and $\gamma_n$ as integral on the left-hand side of \eqref{eqn: intermediate consistency}. 
Note that when $\gamma<1$ and $\gamma\neq0$, since $F$ has a finite mean, both the left hand side and the right hand side of \eqref{eqn: 1st_order_ERV_U} are integrable functions. Due to the local uniformity of \eqref{eqn: 1st_order_ERV_U}, we have
\[
    \lim_{t\to\infty} \frac{1}{y} \int_{0}^{y} \frac{U(\frac{t}{s})-U(t)}{a(t)} ds = \frac{1}{y} \int_{0}^{y} \frac{s^{-\gamma}-1}{\gamma} ds, \quad y>0, \ \gamma<1.
\]
Then,
\begin{equation}
\label{eqn:v-u}
    \lim_{t\to\infty} \frac{V(\frac{y}{t})-U(t)}{a(t)}=\frac{y^{-\gamma}}{\gamma(1-\gamma)}-\frac{1}{\gamma},\quad y>0, \ \gamma<1.
\end{equation}
In particular (\ref{eqn:v-u}) holds with $y=1$. This gives (by subtraction) the equation (\ref{eqn:1st_order_ERV'_V}). The proof of (\ref{eqn:1st_order_ERV'_V}) for the case $\gamma=0$ will be similar.

For $a$ as in \eqref{eqn: 1st_order_ERV_U} and $0<c<1$, we have
\begin{align} \label{eqn: gamma_n - gamma}
    \gamma_n - \gamma &= \int_{(0,1]}\log{\left(\frac{t^\gamma a(n/s_{j+1})}{a(n/m_n)}\right)} \lambda(dt) + \int_{(0,1]} \log{\left(\frac{V(n/s_{\lfloor cj\rfloor+1}) - V(n/s_{j+1})}{\tilde{h}_\gamma(c)a(n/s_{j+1})}\right)}\lambda(dt).
\end{align}
Since $\lim_{n\to\infty} s_{\lceil tm_n\rceil+1}/m_n = t$ for $0<t\leq1$ and $\lim_{j\to\infty} s_{\lceil cj\rceil+1}/s_{j+1}=c$ for $c\in(0,1)$, the integrands of both integrals converge to zero due to the local uniformity in \eqref{eqn: 1st_order_RV} and \eqref{eqn:1st_order_ERV'_V}. It remains to bound them by integrable functions uniformly over $n$. By \cite{segers2005generalized}, there are positive constants $A_1$ and $A_2$ such that for all $n\geq n(\epsilon)$ with $\epsilon>0$ and all $t\in(0,1]$,
\[
    \left|\log{\left(\frac{t^\gamma a(n/s_{j+1})}{a(n/m_n)}\right)}\right| \leq A_1 + A_2\log{(1/t)},
\]
which prove the boundness for the first term in \eqref{eqn: gamma_n - gamma}. For the second term, the conditions on $\left\{s_j\right\}_{j\geq1}$ and the monotonicity of $V$ imply by virtue of the extreme value condition \eqref{eqn: 1st_order_ERV_U} and \eqref{eqn:1st_order_ERV'_V} that
\[
    \limsup_{n\to\infty} \sup_{0<t\leq1} \left| \log{\left(\frac{V(n/s_{\lfloor cj\rfloor+1}) - V(n/s_{j+1})}{\tilde{h}_\gamma(c)a(n/s_{j+1})}\right)} \right| < \infty.
\]
\end{proof}

\section{Proof of Theorem \ref{thm: smooth_tilde_convergence}}
\label{appendix: asymptotic normality}
Let $0<c<1$ and $\lambda\in\Lambda$, and abbreviate $j=\lceil tm_n\rceil$. We have
\begin{align*}
    \sqrt{m}\{\widehat{\gamma}_{n,m}(c,\lambda) - \gamma\} &= \sqrt{m} \biggl\{\int \log(Y_{\lfloor cj \rfloor} - Y_j)\lambda(dt) - \int \log(V(\frac{cm_n}{n})-V(\frac{m_n}{n}))\lambda(dt) - \int\gamma\log(\frac{1}{t})\lambda(dt) \biggr\}\\
    &= \int \sqrt{m} \log\biggl(\frac{t^{\gamma}(Y_{\lfloor cj \rfloor} - Y_j)}{a(\frac{n}{m_n})\tilde{h}_\gamma(c)}\biggr) \lambda(dt).
\end{align*}
Theorem \ref{thm: integrand} below presents a strong approximation of the integrand in the previous equation by a gaussian process, but with the order statistics $Y_j$ replaced by the $V(S_{j+1}/n)$ of approximation \eqref{eqn: Reiss approx2}. Theorem \ref{thm: smooth_tilde_convergence} then follows as an immediate corollary to Theorem \ref{thm: integrand}.

\begin{theorem} \label{thm: integrand}
Assume Conditions \ref{eqn: 1st_order_ERV_U} and \ref{cond2} hold. Let $0<\epsilon<1/2$ and let $0<c_0\leq c_n<1$ be such that $m_n^\eta (1-c_n)\rightarrow\infty$ for some $0<\eta<\epsilon/(2+4\epsilon)$. On a suitable probability space, there exists r.v. $\{S_j,j\geq1\}$ as in \ref{eqn: expo_sum} and a standard Wiener process $W$ such that wp1
\[
    \sup_{0<t\leq1,\ c_0\leq c\leq c_n} |f_n(t;c,\gamma)-g_n(t;c,\gamma,\rho,r)|\rightarrow 0,
\]
as $n\to\infty$, where, abbreviating $j=\lceil tm_n\rceil$,
\[
    f_n(t;c,\gamma)=\sqrt{m_n} \log \biggl(\frac{t^\gamma\{V(S_{\lfloor cj\rfloor+1}/n) - V(S_{j+1}/n)\}}{a(n/m_n)\tilde{h}_\gamma(c)}\biggr)
\]
\[
    g_n(t;c,\gamma,\rho,r)=\frac{t^\gamma \Delta_c\{t^{-\gamma-1}W_n(t)+rH_{\gamma,\rho}(t)\}}{\tilde{h}_\gamma(c)}
\]
and $W_n(t)=-m_n^{-1/2}W(tm_n)$ is also a standard Wiener process.
\end{theorem}
To study the function $f_n$ of interest, decompose it as 
\[
    f_n(t;c,\gamma)=\sqrt{m_n}(\text{I}_n+\text{II}_n+\text{III}_n+\text{IV}_n),
\]
where, with the convenient abbreviation $j=\lceil tm_n\rceil$,
\[
    \exp\text{I}_n=\frac{a(n/S_{j+1})}{a(n/m_n)} \biggl(\frac{S_{j+1}}{m_n}\biggr)^\gamma,
\]
\[
    \exp\text{II}_n=\frac{\tilde{h}(\frac{S_{\lfloor cj\rfloor+1}}{S_{j+1}};\frac{n}{S_{j+1}})}{\frac{1}{c}\int_0^c h_\gamma(\frac{S_{\lfloor uj\rfloor+1}}{S_{j+1}})du - \int_0^1 h_\gamma(\frac{S_{\lfloor uj\rfloor+1}}{S_{j+1}})du},
\]
\[
    \exp\text{III}_n=\biggl(\frac{tm_n}{S_{j+1}}\biggr)^\gamma,
\]
\[
    \exp\text{IV}_n=\frac{\frac{1}{c}\int_0^c h_\gamma(\frac{S_{\lfloor uj\rfloor+1}}{S_{j+1}})du - \int_0^1 h_\gamma(\frac{S_{\lfloor uj\rfloor+1}}{S_{j+1}})du}{\tilde{h}_\gamma(c)},
\]
where
\[
    \tilde{h}(y;x) = \frac{V(\frac{y}{x})-V(\frac{1}{x})}{a(x)}=\frac{\frac{1}{y}\int_0^y h(w;x)dw - \int_0^1 h(w;x)dw}{a(x)}.
\]
The terms $\text{I}_n$ and $\text{II}_n$ will lead to a deterministic bias term, while the terms $\text{III}_n$ and $\text{IV}_n$ will lead to a mean-zero Gaussian term. Also, $\text{I}_n$ and $\text{III}_n$ depend on $t$, while $\text{II}_n$ and $\text{IV}_n$ depend on $t$ and $c$. We will devote a proposition to each of the terms separately.

In the proof of the asymptotic normality by \cite{segers2005generalized}, he proved the following propositions.
\begin{prop} \label{prop:part1}
For every $\epsilon>0$, we have wp1,
\[
    \sup_{0<t\leq1} t^\epsilon |m_n^{1/2} \text{I}_n(t) - rh_{\rho}(t)| \rightarrow 0.
\]
\end{prop}
\begin{prop} \label{prop:part2_VaR}
For every $\epsilon>0$ and $0<c_0<1$, we have wp1,
\[
    \sup_{0<t\leq1,\ c_0\leq c<1} t^\epsilon |m_n^{1/2} \tilde{\text{II}}_n(t) - rt^{-\rho}\frac{H_{\gamma,\rho}(c)}{h_{\gamma}(c)}| \rightarrow 0,
\]
where
\[
    \tilde{\text{II}}_n(t) = \log\biggl(\frac{h(\frac{S_{\lfloor cj\rfloor+1}}{S_{j+1}};\frac{n}{S_{j+1}})}{h_\gamma(\frac{S_{\lfloor cj\rfloor+1}}{S_{j+1}})}\biggr)
\]
and
\[
    h(y;x) = \frac{U(\frac{x}{y})-U(x)}{a(x)}.
\]
\end{prop}

\begin{prop} \label{prop:part3}
For every $\epsilon>0$, we have wp1,
\[
    \sup_{0<t\leq1} t^{1/2+\epsilon} \left|m_n^{1/2} \text{III}_n(t) - \gamma\frac{W_n(t)}{t}\right| \rightarrow 0.
\]
\end{prop}
\begin{prop} \label{prop:part4_VaR}
Let $0<\epsilon<1/2$ and $0<c_0\leq c_n<1$ be such that $m_n^\eta(1-c_n)\rightarrow\infty$ as $n\rightarrow\infty$ for some $0<\eta<\epsilon/(2+4\epsilon)$. Then wp1,
\[
    \sup_{0<t\leq1,\ c_0\leq c\leq c_n} t^{1/2+\epsilon} \left|m_n^{1/2}\tilde{\text{IV}}_n(t) - \frac{1}{|h_\gamma(c)|}\biggl[\frac{W_n(ct)}{ct} - \frac{W_n(t)}{t}\biggr]\right|\rightarrow 0,
\]
where
\[
    \tilde{\text{IV}}_n(t) = \frac{h_\gamma(\frac{S_{\lfloor cj\rfloor+1}}{S_{j+1}})}{h_\gamma(c)}.
\]
\end{prop}

Standing on propositions \ref{prop:part2_VaR} and \ref{prop:part4_VaR}, we will devote propositions for the terms $\text{II}_n$ and $\text{IV}_n$. 

\begin{lemma} \label{lemma:a1}
From (\ref{eqn:cond_alternative}), we have for every $0<y_0<1$, 
\[
    \frac{1}{A(x)} \log\biggl(\frac{\tilde{h}(y;x)}{\tilde{h}_\gamma(y)}\biggr) \rightarrow d\frac{\tilde{H}_{\gamma,\rho}(y)}{\tilde{h}_\gamma(y)} \quad \text{as } x\to\infty, \ \text{uniformly in } y_0\leq y<1,
\]
where
\[
    \tilde{H}_{\gamma,\rho}(y)=\frac{1}{y}\int_{0}^{y}H_{\gamma,\rho}(x)dx - \int_{0}^{1}H_{\gamma,\rho}(x)dx,
\]
and
\[
    \tilde{h}_\gamma(y) = \frac{1}{y}\int_0^y h_\gamma(w) dw - \int_0^1 h_\gamma(w) dw
\]
\end{lemma}
\begin{proof}
\cite{segers2005generalized} showed that 
\[
    h(y;x) = \int_y^1 \frac{a(\frac{x}{w})}{a(x)}wdw = \int_y^1 w^{1-\gamma}\{1+A(x)[dh_\rho(w)+r(x,w)]\}dw,
\]
where $r(x,w) \rightarrow 0$ as $x\to\infty$ uniformly in $y_0\leq w\leq1$, and then
\[
    \frac{h(y;x)}{h_\gamma(y)} = 1 + A(x) \biggl[d\frac{H_{\gamma,\rho}(y)}{h_\gamma(y)} + R(x,y)\biggr]
\]
where $R(x,y)\rightarrow0$ as $x\rightarrow\infty$ uniformly in $y_0\leq y\leq1$.

Then, we have
\begin{align*}
    \frac{1}{y}\int_0^y h(u;x) du &= \frac{1}{y}\int_0^y \int_u^1 v^{1-\gamma}\{1+A(x)[dh_\rho(v)+r(x,v)]\}dv du \\
    &= \frac{1}{y}\int_0^y \Bigl\{\frac{u^{-\gamma}-1}{\gamma} + A(x)[dH_{\gamma,\rho}(u)+r(x,u)]\Bigr\}du \\
    &= \frac{y^{1-\gamma}}{\gamma(1-\gamma)} - \frac{1}{\gamma} + A(x) \frac{1}{y}\int_0^y \{dH_{\gamma,\rho}(u)+r(x,u)\}du
\end{align*}
where $r(x,u) \rightarrow 0$ as $x\to\infty$ uniformly in $y_0\leq u\leq1$. Therefore,
\begin{align*}
    \frac{\tilde{h}_(y;x)}{\tilde{h}_\gamma(y)} &= \frac{\frac{1}{y}\int_0^y h(w;x) dw - \int_0^1  h(w;x) dw}{\tilde{h}_\gamma(y)} \\
    &= 1 + A(x)\biggl[d\frac{\tilde{H}_{\gamma,\rho}(y)}{\tilde{h}_\gamma(y)} + R(x,y)\biggr]
\end{align*}
where $R(x,y)\rightarrow0$ as $x\to\infty$ uniformly in $y_0\leq y\leq1$. Since the function $0<y<1 \mapsto \tilde{H}_{\gamma,\rho}(y)/\tilde{h}_\gamma(y)$ is decreasing and converges to 0 as $y\to1$, we can take a Taylor expansion of the logarithm around 1, proving the lemma.
\end{proof}

\begin{prop} \label{prop:part2}
For every $\epsilon>0$ and $0<c_0<1$, we have wp1,
\[
    \sup_{0<t\leq1,\ c_0\leq c<1} t^\epsilon \left|m_n^{1/2} \text{II}_n(t) - rt^{-\rho}\frac{\tilde{H}_{\gamma,\rho}(c)}{\tilde{h}_{\gamma}(c)}\right| \rightarrow 0,
\]
\end{prop}
\begin{proof}
Suppose that $d\neq0$; the proof for $d=0$ is simpler. We have, abbreviating $j=\lceil tm_n\rceil$ and $S(j,c)=S_{\lfloor cj\rfloor+1} / S_{j+1}$,
\begin{align*}
m_n^{1/2} \text{II}_n(t,c) &= dm_n^{1/2}A(n/m_n) \cdot \frac{A(n/S_{j+1})}{A(n/m_n)} \cdot \frac{1}{dA(n/S_{j+1})} \log\biggl(\frac{\tilde{h}(S(j,c);\frac{n}{S_{j+1}})}{\tilde{h}_\gamma(S(j,c))}\biggr) \\
&= P_n \cdot Q_n(t) \cdot R_n(t,c).
\end{align*}
By assumption, $\lim_{n\rightarrow\infty} P_n = r$. Next, we claim that
\[
    Q_n(t) = t^{-\rho} + t^{-\epsilon/2} \delta_n(t),
\]
\[
    R_n(t,c) = \frac{\tilde{H}_{\gamma,\rho}(S(j,c))}{\tilde{h}_\gamma(S(j,c))} + \epsilon_n(t,c),
\]
\[
    \frac{\tilde{H}_{\gamma,\rho}(S(j,c))}{\tilde{h}_\gamma(S(j,c))} = \frac{\tilde{H}_{\gamma,\rho}(c)}{\tilde{h}_\gamma(c)} + t^{-\epsilon/2}\eta_n(t,c),
\]
where $\delta_n(t), \ \epsilon_n(t,c)$, and $\eta_n(t,c)$ converge to 0 as $n\to\infty$ uniformly in $0<t\leq1$ and $c_0\leq c<1$. Since $\rho\leq0$ and the function $0<y<1 \mapsto \tilde{H}_{\gamma,\rho}(y)/\tilde{h}_\gamma(y)$ is decreasing and converges to 0 as $y\to1$, this claim is sufficient to prove the proposition.

The claim for $\delta_n(t)$ is proved by \cite{segers2005generalized}. The claim for $\epsilon_n(t,c)$ follows directly from Lemma \ref{lemma:a1}. Since $S(j,c)\rightarrow c$ as $j\rightarrow\infty$ uniformly in $c_0\leq c\leq1$, the claim for $\eta_n(t,c)$ follows easily.
\end{proof}

\begin{prop} \label{prop:part4}
Let $0<\epsilon<1/2$ and $0<c_0\leq c_n<1$ be such that $m_n^\eta(1-c_n)\rightarrow\infty$ as $n\rightarrow\infty$ for some $0<\eta<\epsilon/(2+4\epsilon)$. Then wp1,
\[
    \sup_{0<t\leq1,\ c_0\leq c\leq c_n} t^{1/2+\epsilon} \left|m_n^{1/2}\text{IV}_n(t) - \biggl[\frac{t^\gamma \Delta_c \{t^{-\gamma-1}W_n(t)\}}{\tilde{h}_\gamma(c)} - \gamma\frac{W_n(t)}{t}\biggr]\right| \rightarrow 0, 
\]
\end{prop}
\begin{proof}
Let $4\eta<\nu<2\epsilon/(1+2\epsilon)$ and $j_n=\lceil k_n^\nu\rceil$. Define $\tilde{\Delta}_x\{\tilde{W}(c)\}=\bigl(\frac{W(x)}{x}-\frac{W(cx)}{cx}\bigr)c^{-\gamma}$ and $\tilde{\Delta}_x\{\tilde{W}_1(c)\}=\bigl(\frac{W(x+1)}{x+1}-\frac{W(\lfloor cx\rfloor+1)}{\lfloor cx\rfloor+1}\bigr)c^{-\gamma}$.
Our quantity of interest is bounded by the maximum of 
\[
    \sup_{0<t\leq j_n/m_n, \ c_0\leq c\leq c_n} t^{1/2+\epsilon} \left|\text{IV}_n(t,c)\right|,
\]
\[
    \sup_{0<t\leq j_n/m_n, \ c_0\leq c\leq c_n} \frac{(tm_n)^{1/2+\epsilon}}{m_n^\epsilon \tilde{h}_\gamma(c)} \left|\frac{1}{c}\int_0^c \tilde{\Delta}_j\{\tilde{W}_1(u)\} du - \int_0^1 \tilde{\Delta}_j\{\tilde{W}_1(u)\} du \right|,
\]
\[
    \sup_{j_n/m_n\leq t\leq1,\ c_0\leq c\leq c_n} m_n^{-\epsilon}(tm_n)^{1/2+\epsilon} \left|\text{IV}_n(t,c) - \frac{1}{\tilde{h}_\gamma(c)}\biggl[\frac{1}{c}\int_0^c \tilde{\Delta}_j\{\tilde{W}_1(u)\} du - \int_0^1 \tilde{\Delta}_j\{\tilde{W}_1(u)\} du\biggr]\right|,
\]
\begin{align*}
    \sup_{0<t\leq1,\ c_0\leq c\leq c_n} & \frac{(tm_n)^{1/2+\epsilon}}{m_n^\epsilon\tilde{h}_\gamma(c)} \bigg|\biggl[\frac{1}{c}\int_0^c \tilde{\Delta}_j\{\tilde{W}_1(u)\} du - \int_0^1 \tilde{\Delta}_j\{\tilde{W}_1(u)\} du\biggr] \\
    & \quad - \biggl[\frac{1}{c}\int_0^c \tilde{\Delta}_{tm_n}\{\tilde{W}(u)\} du - \int_0^1 \tilde{\Delta}_{tm_n}\{\tilde{W}(u)\} du\biggr]\bigg|,
\end{align*}
which we call $R_{n,1},\ R_{n,2},\ R_{n,3}$, and $R_{n,4}$ respectively. Recall that $\{\xi_i\}_{i=1}^j$ are $j$ independent standard exponential r.v.'s. We will use the following conclusions proved by \cite{segers2005generalized} from (\ref{eqn:serger1}) to (\ref{eqn:serger5}):
\begin{equation} \label{eqn:serger1}
    \max_{1\leq j \leq j_n} \left|\log{(\xi_{j+1}/j)}\right| = O(\log{j_n}) = O(\log{m_n}) \quad n\to\infty,
\end{equation}
% \begin{equation} \label{eqn:serger2}
%     \sup_{c_0 \leq c \leq c_n} |\log h_\gamma(c)| = O(\log m_n), \quad n\to\infty.
% \end{equation}
\begin{equation} \label{eqn:serger3}
    \sup_{0<t\leq j_n/m_n, \ c_0\leq c\leq c_n} \frac{(tm_n)^{1/2+\epsilon}}{m_n^\epsilon |h_\gamma(c)|} \left|\tilde{W}_j(c) \right| = O\biggl(\frac{j_n^\epsilon(\log j_n)^{1/2}}{k_n^\epsilon(1-c_n)}\biggr), \quad n\to\infty,
\end{equation}
and
\begin{equation} \label{eqn:serger5}
    \sup_{x>0,\ c_0\leq c\leq1} x^{1/2+\epsilon} \left|\frac{W(\lfloor c\lceil x\rceil\rfloor+1)}{\lfloor c\lceil x\rceil\rfloor+1} - \frac{W(cx)}{cx}\right| < \infty \quad \text{wp1}.
\end{equation}
\\ \\
Let $S(j,c)=S_{\lfloor cj\rfloor+1} / S_{j+1}$. First of all, we have
\[
    R_{n,1} \leq \frac{j_n^{1/2+\epsilon}}{k_n^\epsilon} \sup_{c_0\leq c\leq c_n} \biggl\{\max_{1\leq j \leq j_n} \left|\log\Bigl[\frac{1}{c} \int_0^c h_\gamma(S(j,u)) du - \int_0^1 h_\gamma(S(j,u)) du\Bigr] \right| + \left|\log\tilde{h}_\gamma(c) \right|\biggr\}.
\]
Now on the one hand, $\sup\{S(j,c):j\geq1,c_0\leq c<1\}<\infty$, while on the other hand $S(j,c)\geq1+\xi_{j+1}/S_j$ for all $0<c<1$. Hence as $n\to\infty$, the asymptotic equivalence $\frac{1}{x}\int_0^x h_\gamma(y) dy \sim x-1$ for $x\to1$ implies
\[
    \max_{1\leq j \leq j_n} \left|\log{\tilde{h}_\gamma(S(j,c))}\right| = O\left(\max_{j=1,\ldots,j_n} \left|\log{(\xi_{j+1}/j)}\right|\right).
\]
From \eqref{eqn:serger1}, it can implied that
\[
    \max_{1\leq j \leq j_n} \left|\log{\tilde{h}_\gamma(S(j,c))}\right| = O(\log m_n), \quad n\to\infty.
\]
Besides, the same order of magnitude is also found in
\[
    \sup_{c_0 \leq c \leq c_n} |\log \tilde{h}_\gamma(c)| = O(\log m_n), \quad n\to\infty.
\]
Together, we have wp1,
\[
    R_{n,1} = O\biggl(\frac{j_n^{1/2+\epsilon} \log m_n}{m_n^\epsilon}\biggr) = O\biggl(\frac{\log m_n}{m_n^{\epsilon-\nu(1/2+\epsilon)}}\biggr), \quad n\to\infty,
\]
by the choice of $\nu$.

Secondly, we can get
\[
    R_{n,2} = O\biggl(\frac{j_n^\epsilon(\log j_n)^{1/2}}{k_n^\epsilon(1-c_n)}\biggr), \quad n\to\infty,
\]
from (\ref{eqn:serger3}). However, since $\epsilon(1-\nu) > \epsilon/(1+2\epsilon)$, we have $R_{n,2} \to 0$ by the choice of $c_n$.

Thirdly, the term $R_{n,3}$ is bounded by
\[
    \sup_{j\geq j_n, \ c_0\leq c\leq c_n} j^{1/2} \left|\text{IV}_n(t,c) - \frac{1}{\tilde{h}_\gamma(c)}\biggl[\frac{1}{c}\int_0^c \tilde{\Delta}_j\{\tilde{W}_1(u)\} du - \int_0^1 \tilde{\Delta}_j\{\tilde{W}_1(u)\} du\biggr]\right|.
\]
Since
\[
    \frac{(1-c_n)^2 j_n^{1/2}}{\log j_n} \sim \frac{(1-c_n)^2 m_n^{\nu/2}}{\nu \log m_n} \rightarrow \infty, \quad n\to\infty,
\]
our job for $R_{n,3}$ will be accomplished if we can show that whenever $0<c_0\leq c_j^\prime<1$ is such that $(1-c_j^\prime)^2j^{1/2} / \log j \rightarrow \infty$ as $j\rightarrow\infty$, then also wp1,
\[
    \sup_{c_0\leq c\leq c_j^\prime} j^{1/2} \left|\text{IV}_n(t,c) - \frac{1}{\tilde{h}_\gamma(c)}\biggl[\frac{1}{c}\int_0^c \tilde{\Delta}_j\{\tilde{W}_1(u)\} du - \int_0^1 \tilde{\Delta}_j\{\tilde{W}_1(u)\} du\biggr]\right| \rightarrow 0.
\]
Denote $S_j = j+W(j)+Z_j$. With the properties of the Hungarian construction, \cite{segers2005generalized} showed that, wp1, 
\[
    h_\gamma(\frac{S(j,c)}{c}) = \frac{W(j+1)}{j+1} - \frac{W(\lfloor cj\rfloor+1)}{\lfloor cj\rfloor+1} + O(j^{-1}\log j), \quad j\to\infty,
\]
uniformly in $c_0\leq c\leq1$. \\ \\
Thus, since the function $h_\gamma$ satisfies the functional equations
\[
    h_\gamma(xy) = x^{-\gamma}h_\gamma(y) + h_\gamma(x) \quad \text{and} \quad x^{\gamma}h_\gamma(x) = -h_\gamma(1/x), \quad \text{for } x,y>0,
\]
we can get
\begin{align*}
    & \quad \frac{1}{c} \int_0^c h_\gamma(S(j,u)) du - \int_0^1 h_\gamma(S(j,u)) du - \Bigl[\frac{1}{c} \int_0^c h_\gamma(u) du - \int_0^1 h_\gamma(u) du\Bigr]\\
    &= \frac{1}{c} \int_0^c h_\gamma(\frac{S(j,u)}{u})u^{-\gamma} du - \int_0^1 h_\gamma(\frac{S(j,u)}{u})u^{-\gamma} du \\
    &= \frac{1}{c} \int_0^c \biggl(\frac{W(j+1)}{j+1} - \frac{W(\lfloor uj\rfloor+1)}{\lfloor uj\rfloor+1}\biggr)u^{-\gamma} du  \\
    & \qquad - \int_0^1 \biggl(\frac{W(j+1)}{j+1} - \frac{W(\lfloor uj\rfloor+1)}{\lfloor uj\rfloor+1}\biggr)u^{-\gamma} du + O(j^{-1}\log j) \\
    &= -\biggl(\frac{1}{c} \int_0^c \frac{W(\lfloor uj\rfloor+1)}{\lfloor uj\rfloor+1}u^{-\gamma} du - \int_0^1 \frac{W(\lfloor uj\rfloor+1)}{\lfloor uj\rfloor+1}u^{-\gamma} du \biggr) \\
    & \qquad + \frac{c^{-\gamma}-1}{1-\gamma} \frac{W(j+1)}{j+1} + O(j^{-1}\log j), \quad j\to\infty,
\end{align*}
uniformly in $c_0\leq c\leq1$. Then, repeat the argument for $\log(1+x)$ as $x\to0$ and we can get
\begin{align*}
    \text{IV}_n(t,c) &= \log\biggl(1 + \frac{\frac{1}{c} \int_0^c h_\gamma(S(j,u)/u)u^{-\gamma} du - \int_0^1 h_\gamma(S(j,u)/u)u^{-\gamma} du}{\tilde{h}_\gamma(c)} \biggr)\\
    &= -\frac{1}{\tilde{h}_\gamma(c)}\biggl[\frac{1}{c} \int_0^c \frac{W(\lfloor uj\rfloor+1)}{\lfloor uj\rfloor+1}u^{-\gamma} du - \int_0^1 \frac{W(\lfloor uj\rfloor+1)}{\lfloor uj\rfloor+1}u^{-\gamma} du \biggr] \\
    & \qquad + \gamma\frac{W(j+1)}{j+1} + O\biggl(\frac{\log j}{j(1-c_j^\prime)^2}\biggr), \quad j\to\infty,
\end{align*}
uniformly in $c_0\leq c\leq c_j^\prime$. The remainder term is $o(j^{-1/2})$ by assumption.

Fourthly and finally, the term $R_{n,4}$ is not larger than
\begin{align*}
    \frac{1}{k_n^\epsilon \tilde{h}_\gamma(c_n)} & \sup_{x>0,\ c\in[c_0,1]} x^{1/2+\epsilon} \bigg|\biggl[\frac{1}{c}\int_0^c \biggl(\frac{W(\lceil x\rceil+1)}{\lceil x\rceil+1} - \frac{W(\lfloor u\lceil x\rceil\rfloor+1)}{\lfloor u\lceil x\rceil\rfloor+1}\biggr)u^{-\gamma} du \\
    & - \int_0^1 \biggl(\frac{W(\lceil x\rceil+1)}{\lceil x\rceil+1} - \frac{W(\lfloor u\lceil x\rceil\rfloor+1)}{\lfloor u\lceil x\rceil\rfloor+1}\biggr)u^{-\gamma} du \biggr] - \biggl[\frac{1}{c}\int_0^c \biggl(\frac{W(x)}{x} - \frac{W(ux)}{ux}\biggr) du \\
    & - \int_0^1 \biggl(\frac{W(x)}{x} - \frac{W(cx)}{cx}\biggr) du\biggr]\bigg|.
\end{align*}
Again by assumption, $k_n^\epsilon \tilde{h}_\gamma(c_n) \to \infty$, so it remains to show that the supremum is finite wp1. The latter is readily reduced to the assertion, wp1, 
\begin{equation} \label{eqn:assertion}
    \sup_{x>0,\ c_0\leq c\leq1} x^{1/2+\epsilon} \left|\frac{1}{c}\int_0^c \biggl(\frac{W(\lfloor u\lceil x\rceil\rfloor+1)}{\lfloor u\lceil x\rceil\rfloor+1} - \frac{W(ux)}{ux}\biggr) du - \int_0^1 \biggl(\frac{W(\lfloor u\lceil x\rceil\rfloor+1)}{\lfloor u\lceil x\rceil\rfloor+1} - \frac{W(ux)}{ux}\biggr) du\right| < \infty.
\end{equation}
From (\ref{eqn:serger5}) the assertion (\ref{eqn:assertion}) follows.
\end{proof}

\begin{proof} [\textbf{Proof of Theorem \ref{thm: integrand}}]
After the laborious estimates of the previous pages, the pieces of the puzzle fall together. Collect the approximations of Propositions \ref{prop:part1}, \ref{prop:part2}, \ref{prop:part3}, and \ref{prop:part4} to get
\[
    \sup_{0<t\leq1,\ c_0\leq c\leq c_n} t^{1/2+\epsilon} \left|f_n(t;c,\gamma)-\bar{g}_n(t;c,\gamma,\rho,r)\right| \rightarrow 0 \quad \text{wp1},
\]
where
\[
    \bar{g}_n(t;c,\gamma,\rho,r) = r\biggl(h_\rho(t)+t^{-\rho}\frac{\tilde{H}_{\gamma,\rho}(c)}{\tilde{h}_{\gamma}(c)}\biggr) + \frac{t^\gamma \Delta_c\{t^{-\gamma-1}W_n(t)\}}{\tilde{h}_\gamma(c)}.
\]
The functional relation
\[
    H_{\gamma,\rho}(xy) = x^{-(\gamma+\rho)}H_{\gamma,\rho}(y) + H_{\gamma,\rho}(x) + x^{-\gamma}h_\gamma(y)h_\rho(x), \quad \text{for} \ x,y>0
\]
implies
\[
    \Delta_c\{H_{\gamma,\rho}(t)\} = t^{-(\gamma+\rho)} \biggl[\frac{1}{c}\int_0^c H_{\gamma,\rho}(u)du - \int_0^1 H_{\gamma,\rho}(u)du\biggr] + t^{-\gamma} h_\rho(t) \tilde{h}_\gamma(c)
\]
which confirms $\bar{g}_n = g_n$.
\end{proof}

\section{Simulation Results} \label{appendix: simulations}
% student t (v=4)
\begin{figure}[H]
    \begin{subfigure}{0.5\textwidth}
        \includegraphics[width=\linewidth]{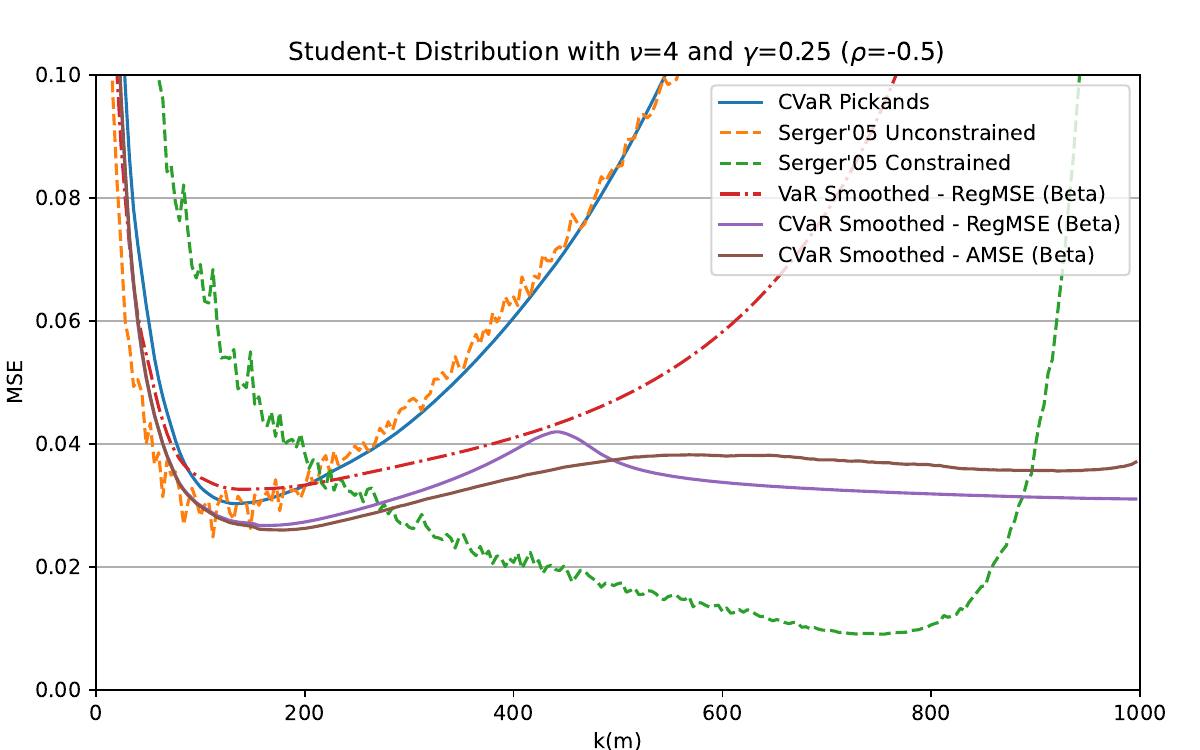}
    \end{subfigure}\hfill
    \begin{subfigure}{0.5\textwidth}
        \includegraphics[width=\linewidth]{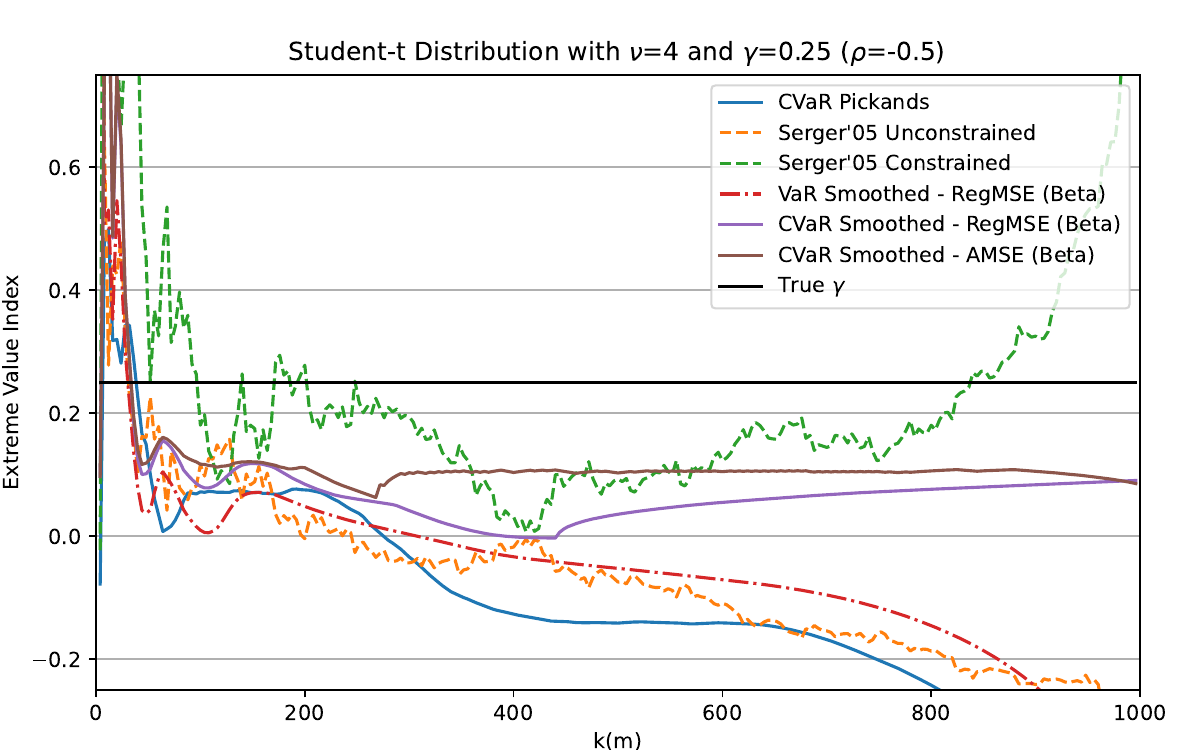}
    \end{subfigure}
    \caption{The samples are from the Student's t distribution with $\nu=4\ (\gamma=0.25)$, and others are the same as Figure \ref{fig: burr0.25}}
    \label{fig: t0.25}
\end{figure}

% Expon - 0
\begin{figure}[H]
    \begin{subfigure}{0.5\textwidth}
        \includegraphics[width=\linewidth]{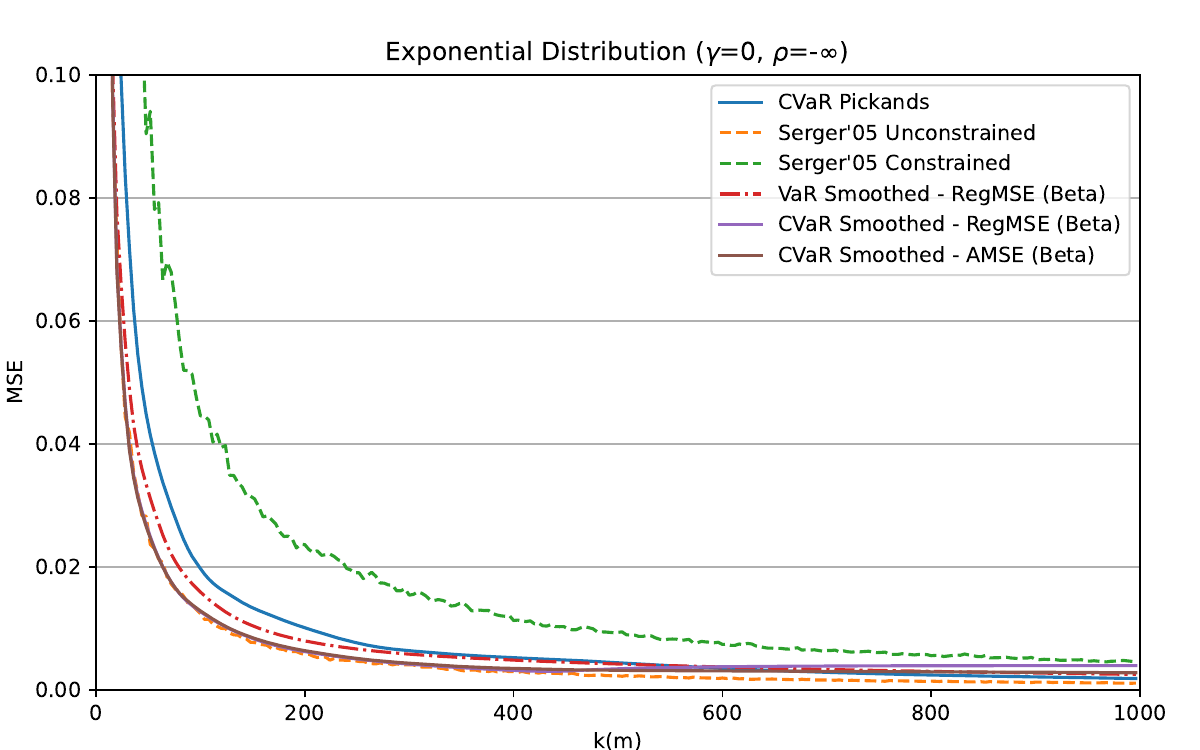}
    \end{subfigure}\hfill
    \begin{subfigure}{0.5\textwidth}
        \includegraphics[width=\linewidth]{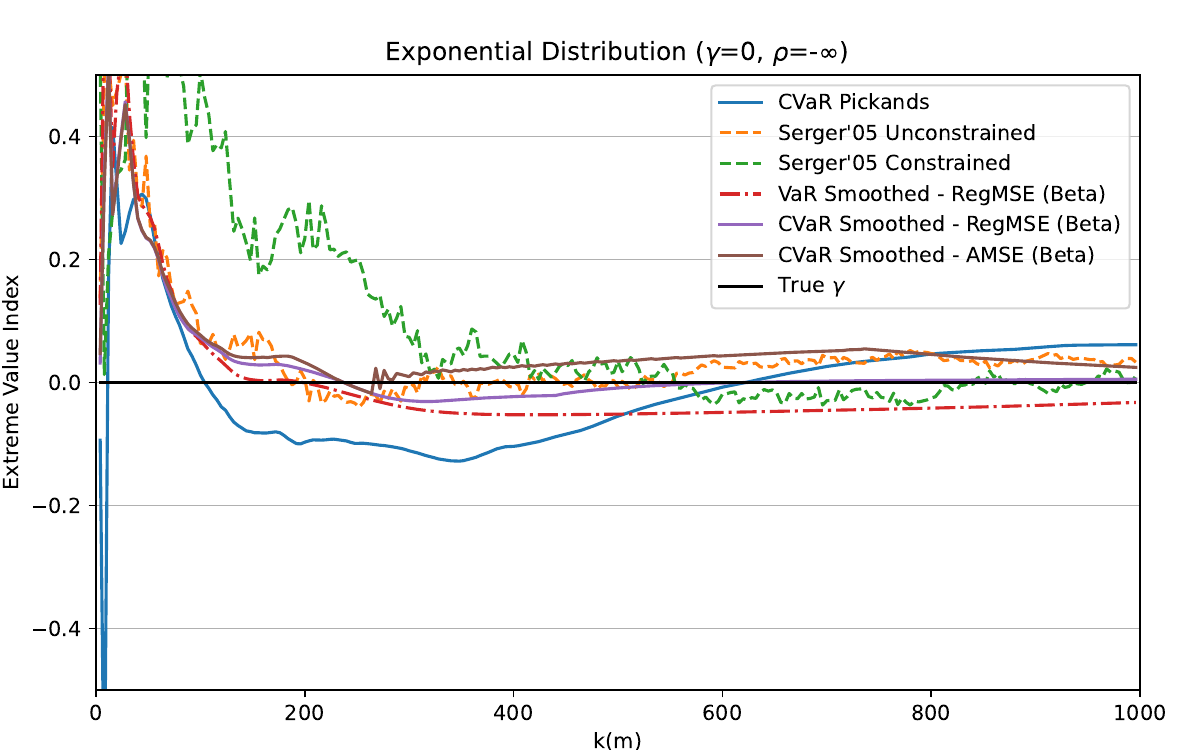}
    \end{subfigure}
    \caption{The samples are from the exponential distribution, and others are the same as Figure \ref{fig: burr0.25}}
    \label{fig: expon}
\end{figure}

% Gumbel - 0
\begin{figure}[H]
    \begin{subfigure}{0.5\textwidth}
        \includegraphics[width=\linewidth]{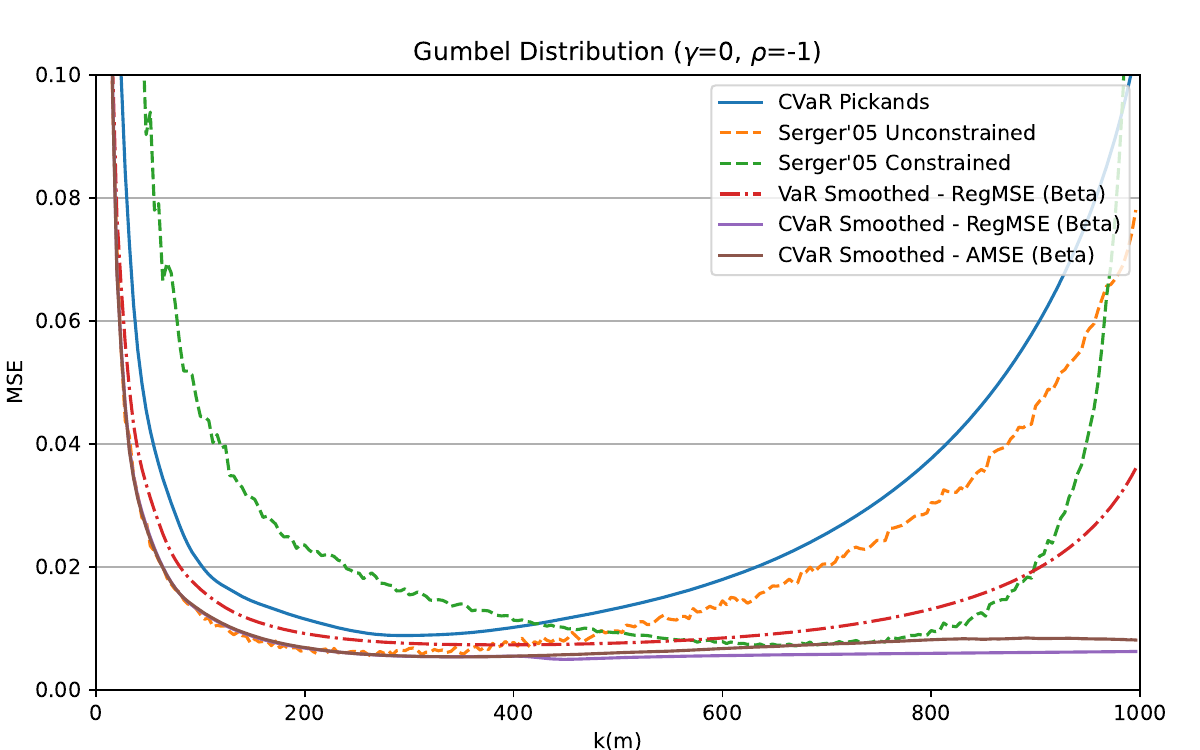}
    \end{subfigure}\hfill
    \begin{subfigure}{0.5\textwidth}
        \includegraphics[width=\linewidth]{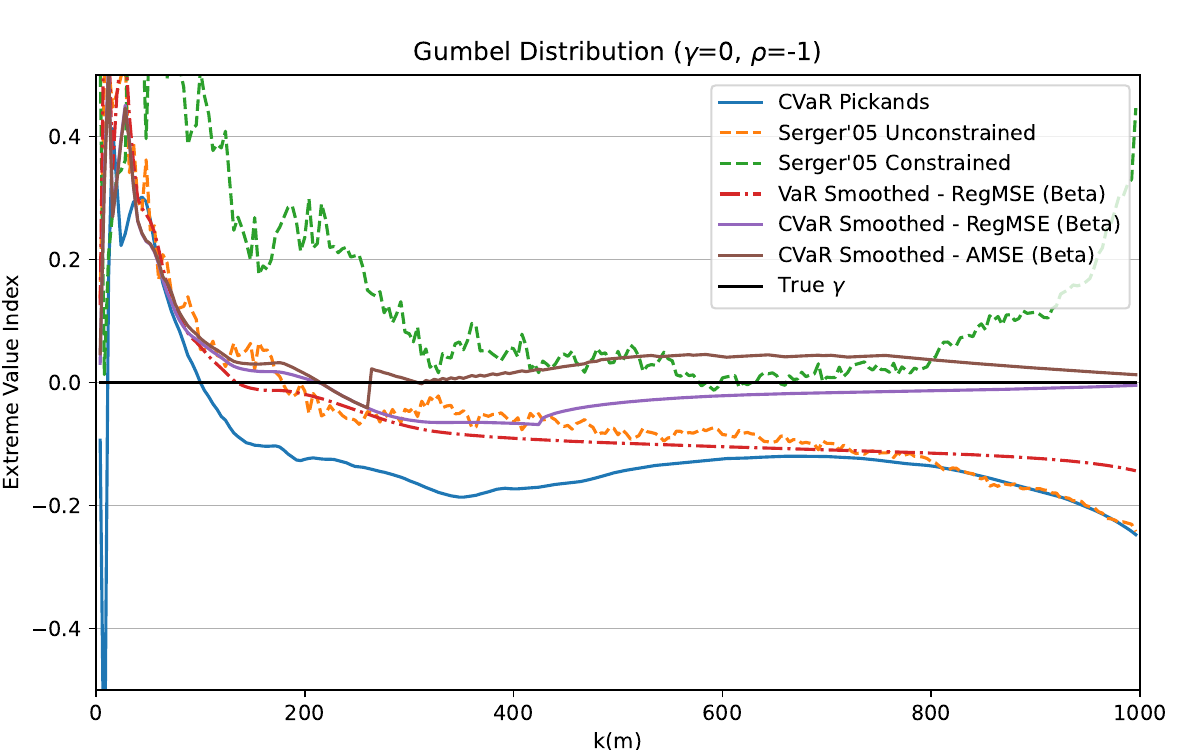}
    \end{subfigure}
    \caption{The samples are from the Gumbel distribution, and others are the same as Figure \ref{fig: burr0.25}}
    \label{fig: gumbel}
\end{figure}

% Logistic - 0
\begin{figure}[H]
    \begin{subfigure}{0.5\textwidth}
        \includegraphics[width=\linewidth]{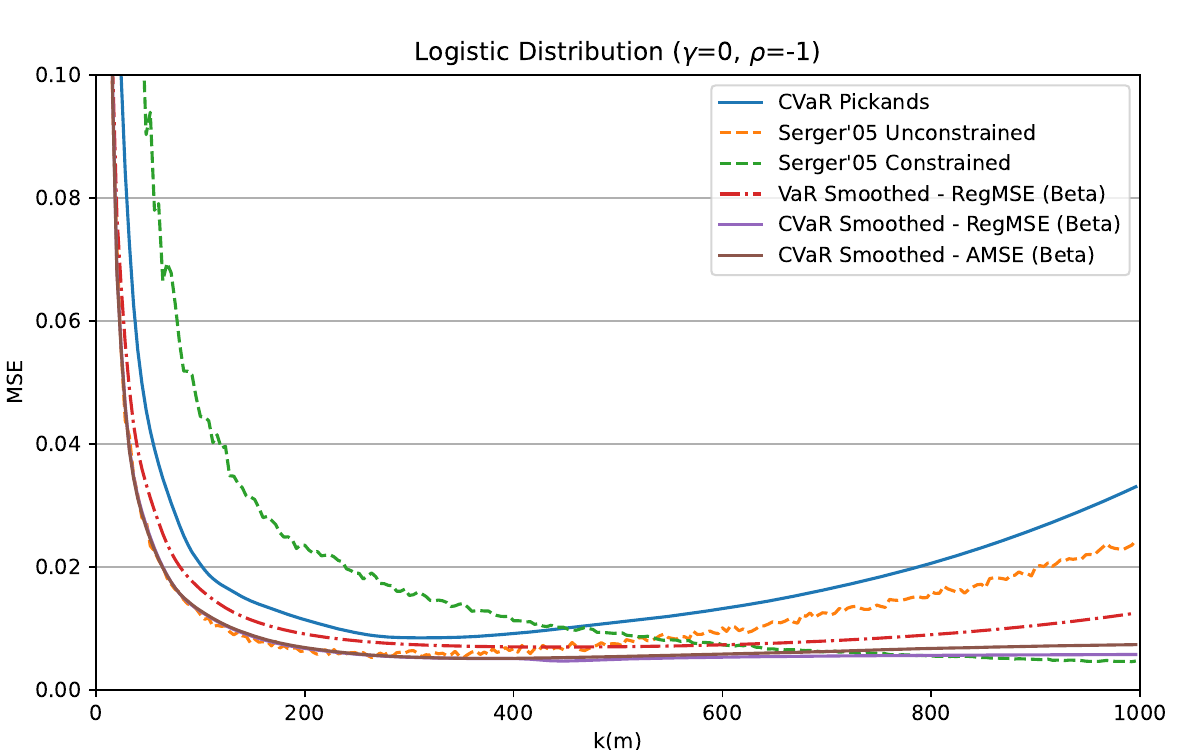}
    \end{subfigure}\hfill
    \begin{subfigure}{0.5\textwidth}
        \includegraphics[width=\linewidth]{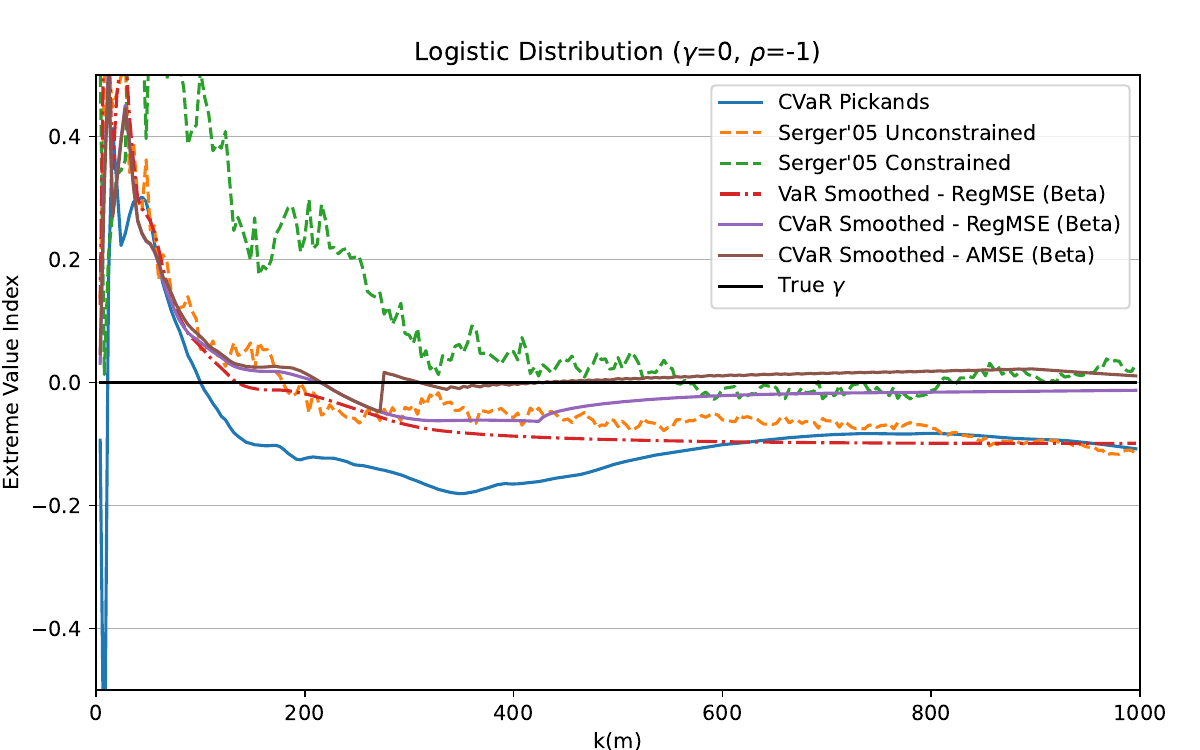}
    \end{subfigure}
    \caption{The samples are from the logistic distribution, and others are the same as Figure \ref{fig: burr0.25}}
    \label{fig: logistic}
\end{figure}

% Weibull - 0
\begin{figure}[H]
    \begin{subfigure}{0.5\textwidth}
        \includegraphics[width=\linewidth]{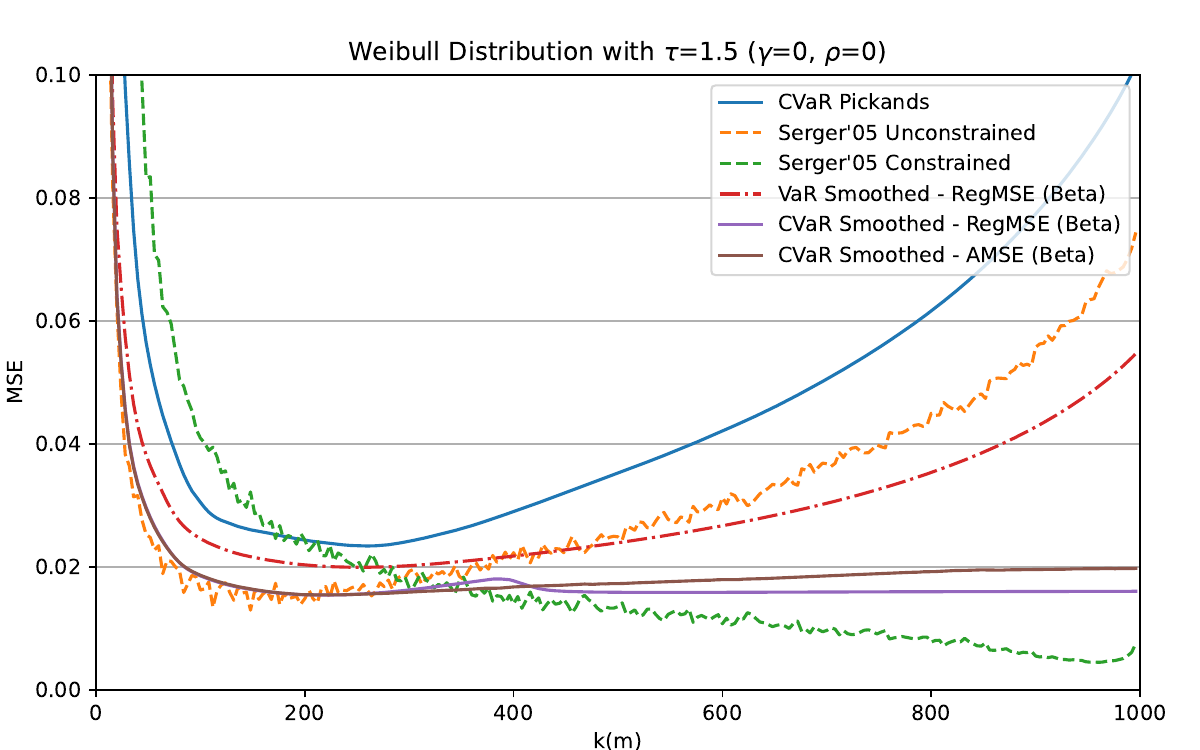}
    \end{subfigure}\hfill
    \begin{subfigure}{0.5\textwidth}
        \includegraphics[width=\linewidth]{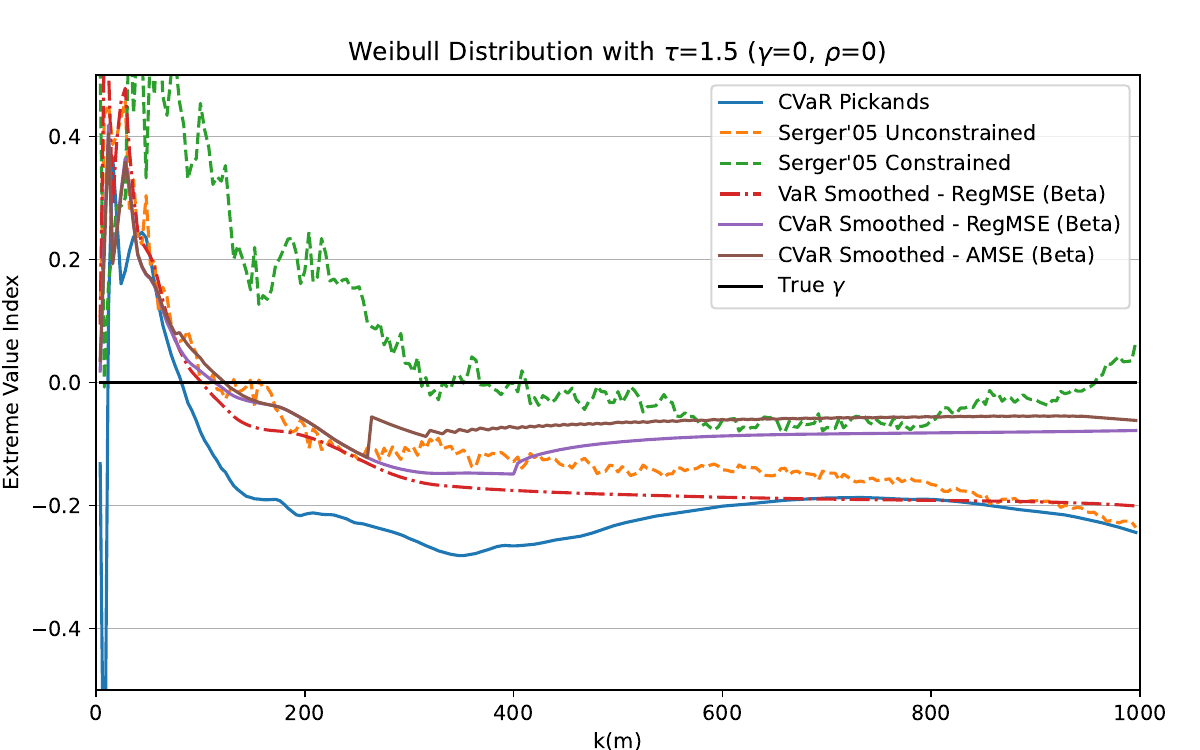}
    \end{subfigure}
    \caption{The samples are from the Weibull distribution with $\tau=1.5$, and others are the same as Figure \ref{fig: burr0.25}}
    \label{fig: weibull}
\end{figure}

% Lognormal - 0
\begin{figure}[H]
    \begin{subfigure}{0.5\textwidth}
        \includegraphics[width=\linewidth]{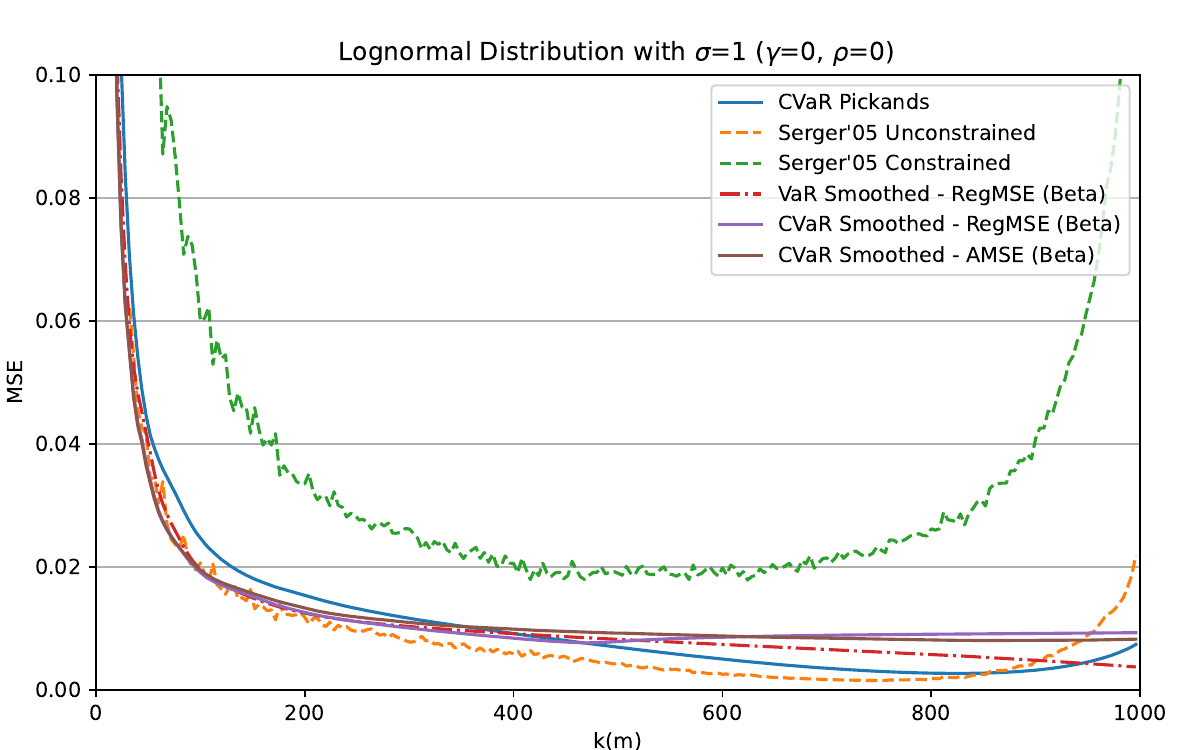}
    \end{subfigure}\hfill
    \begin{subfigure}{0.5\textwidth}
        \includegraphics[width=\linewidth]{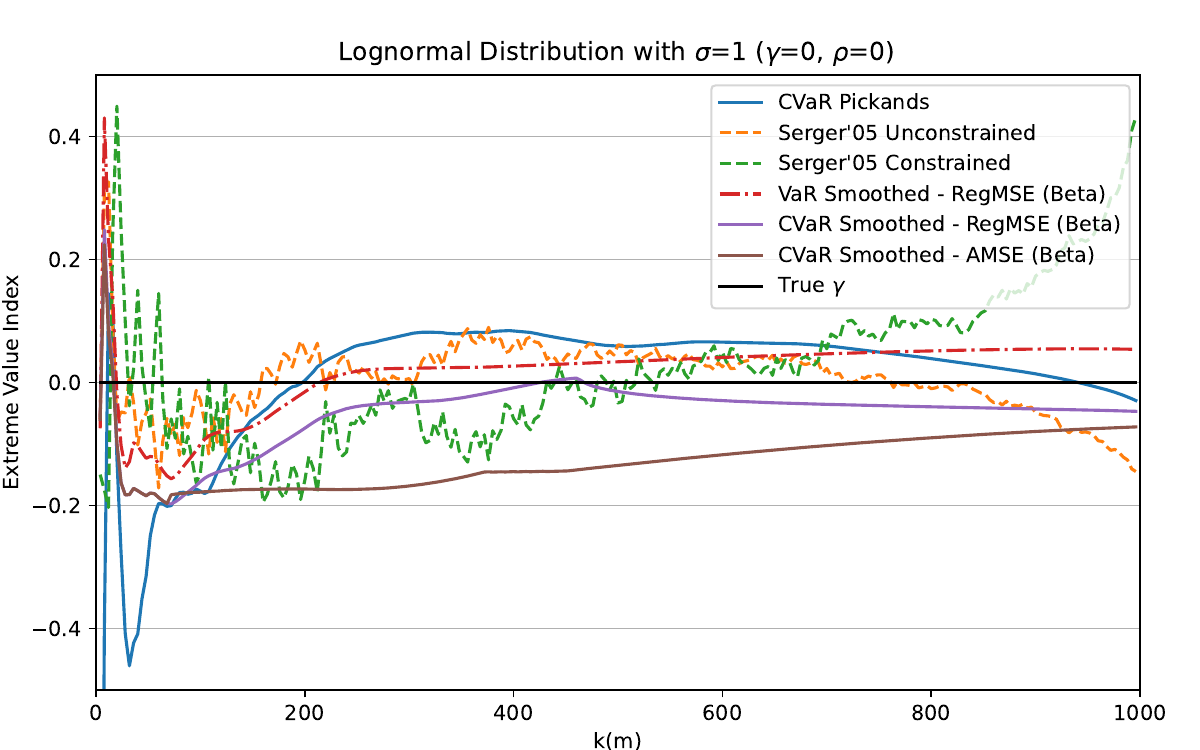}
    \end{subfigure}
    \caption{The samples are from the lognormal distribution with $\sigma=1$, and others are the same as Figure \ref{fig: burr0.25}}
    \label{fig: lognorm}
\end{figure}

% Normal - 0
\begin{figure}[H]
    \begin{subfigure}{0.5\textwidth}
        \includegraphics[width=\linewidth]{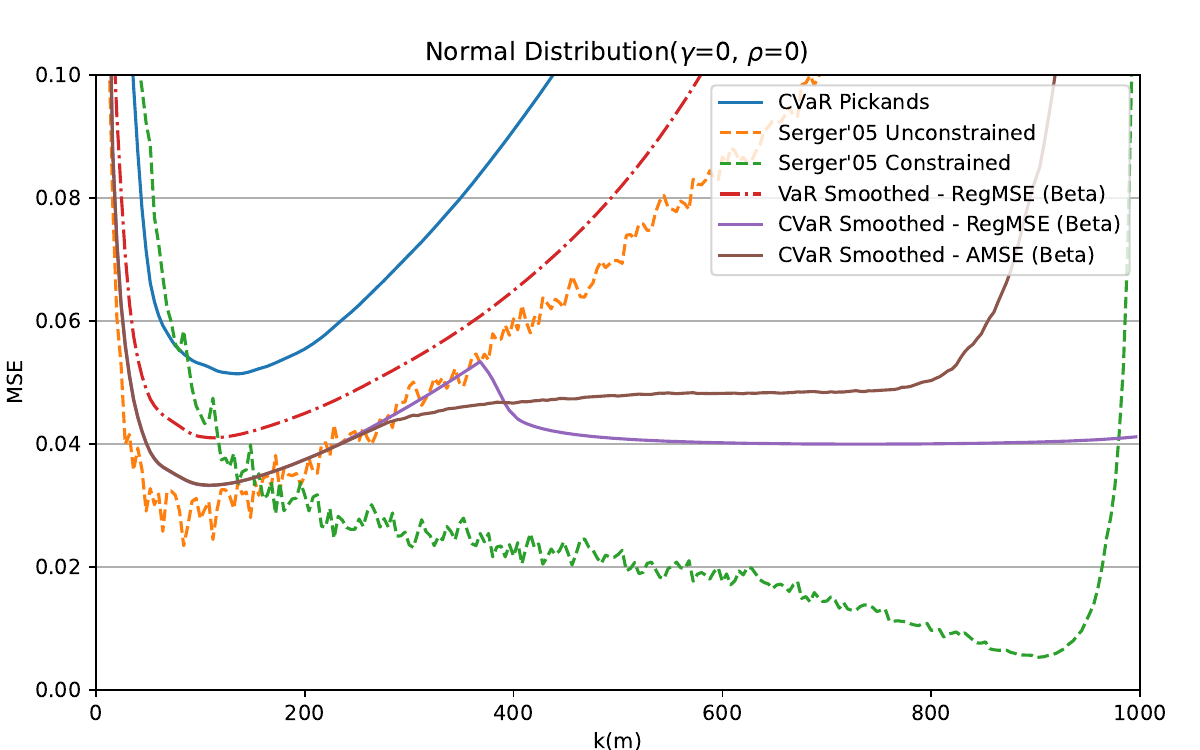}
    \end{subfigure}\hfill
    \begin{subfigure}{0.5\textwidth}
        \includegraphics[width=\linewidth]{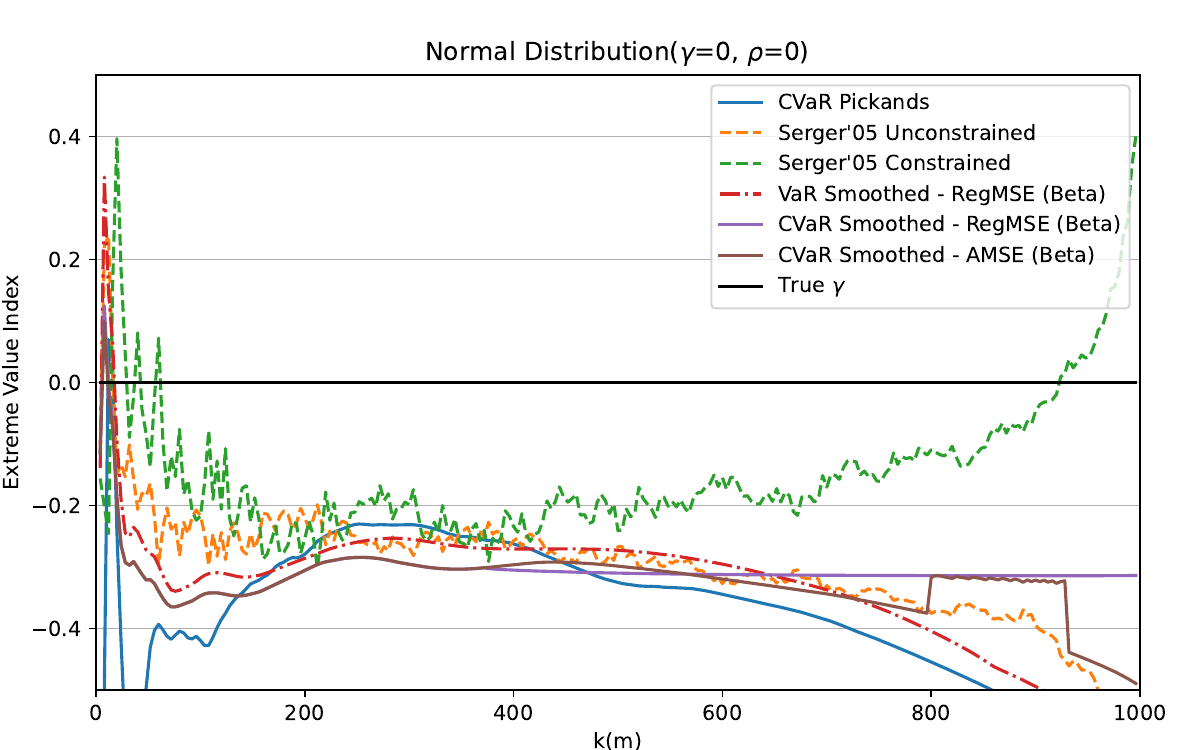}
    \end{subfigure}
    \caption{The samples are from the standard normal distribution, and others are the same as Figure \ref{fig: burr0.25}}
    \label{fig: norm}
\end{figure}

% GEV -0.2
\begin{figure}[H]
    \begin{subfigure}{0.5\textwidth}
        \includegraphics[width=\linewidth]{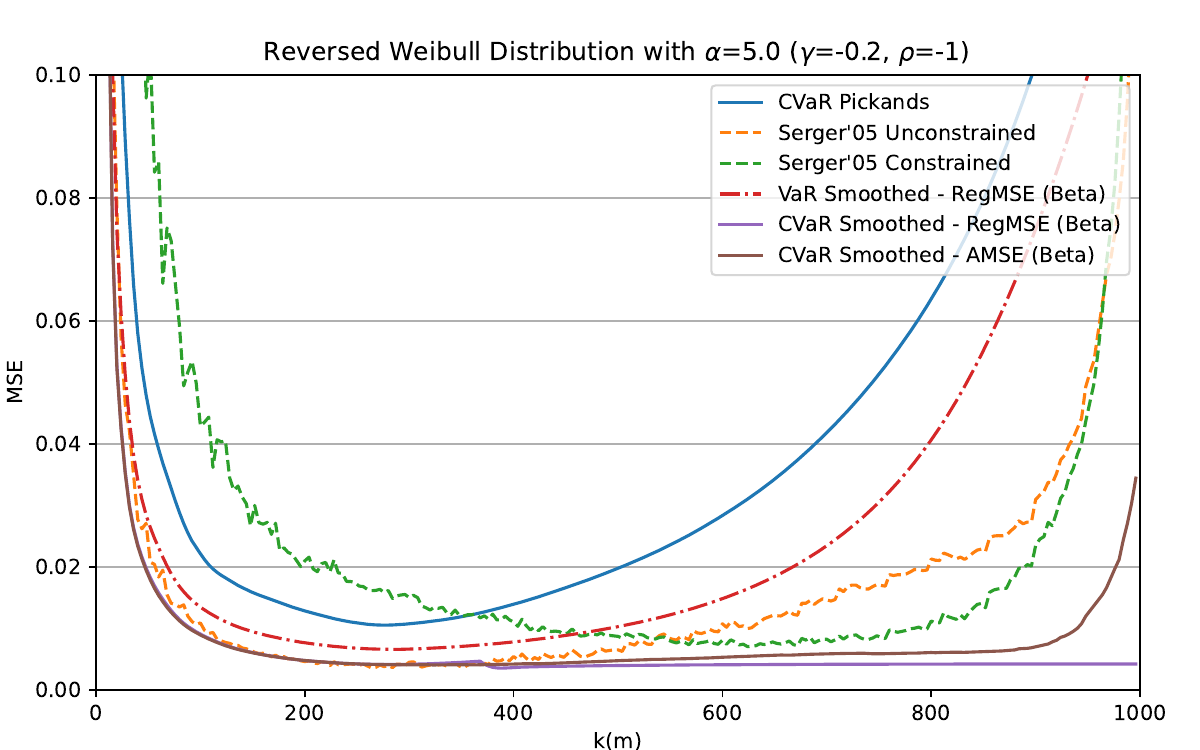}
    \end{subfigure}\hfill
    \begin{subfigure}{0.5\textwidth}
        \includegraphics[width=\linewidth]{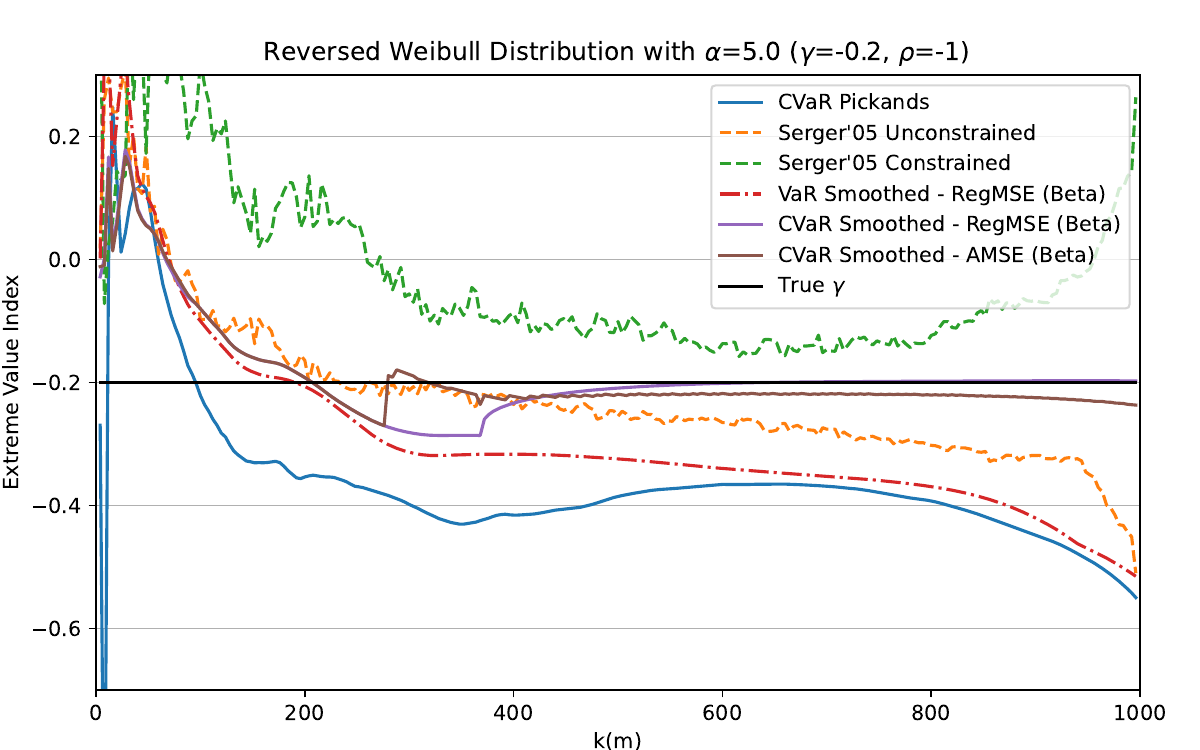}
    \end{subfigure}
    \caption{The samples are from the reversed Weibull distribution with $\gamma=-0.2$, and others are the same as Figure \ref{fig: burr0.25}}
    \label{fig: rev weibull-0.2}
\end{figure}

% GPD -0.2
\begin{figure}[H]
    \begin{subfigure}{0.5\textwidth}
        \includegraphics[width=\linewidth]{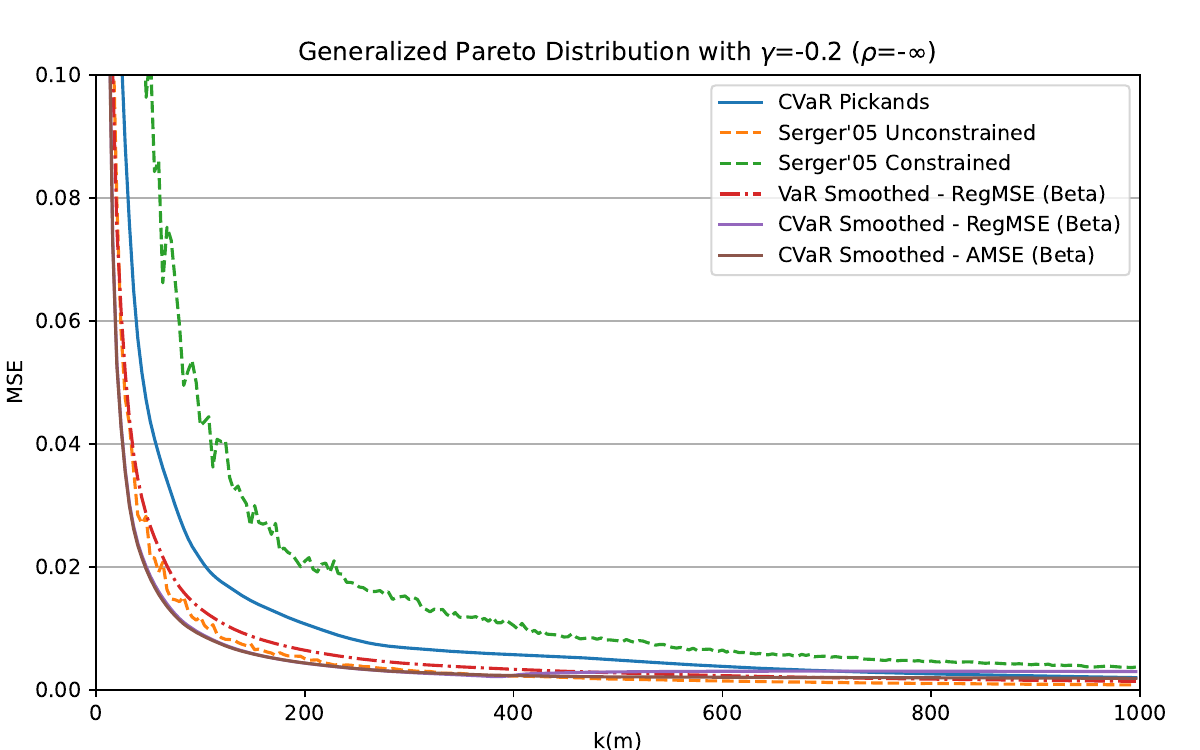}
    \end{subfigure}\hfill
    \begin{subfigure}{0.5\textwidth}
        \includegraphics[width=\linewidth]{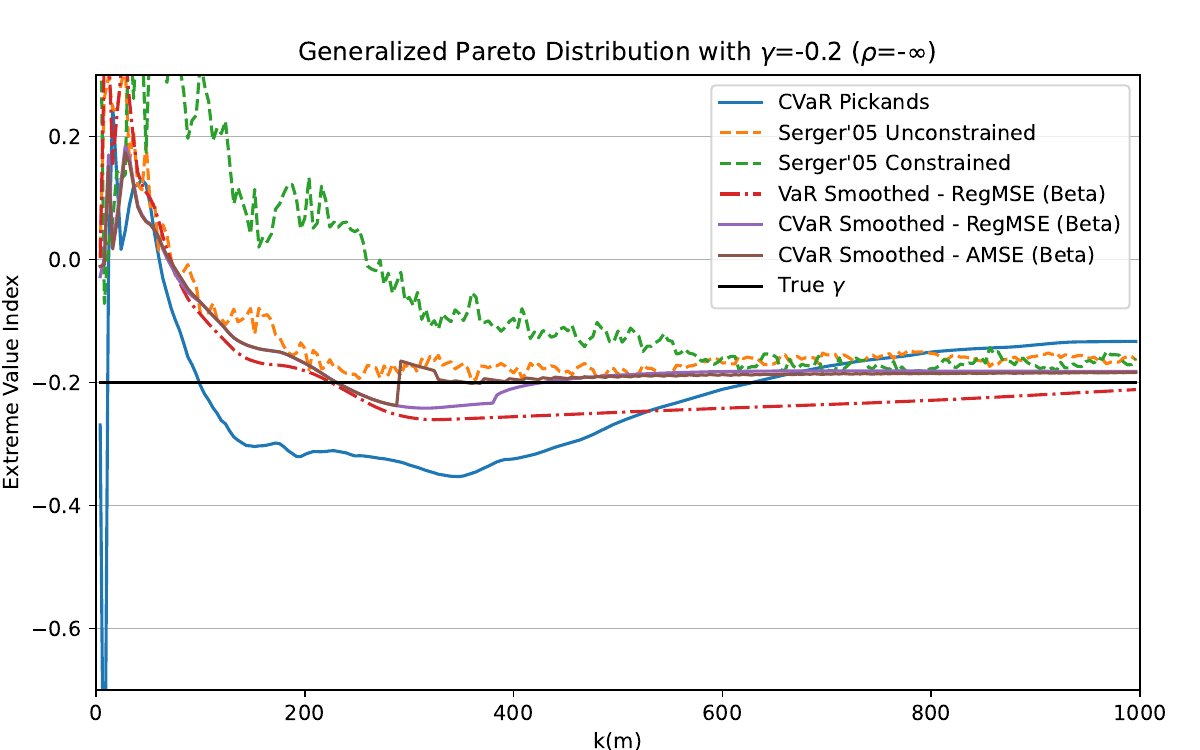}
    \end{subfigure}
    \caption{The samples are from the generalized Pareto distribution with $\gamma=-0.2$, and others are the same as Figure \ref{fig: burr0.25}}
    \label{fig: gpd-0.2}
\end{figure}

% Uniform - -1
\begin{figure}[H]
    \begin{subfigure}{0.5\textwidth}
        \includegraphics[width=\linewidth]{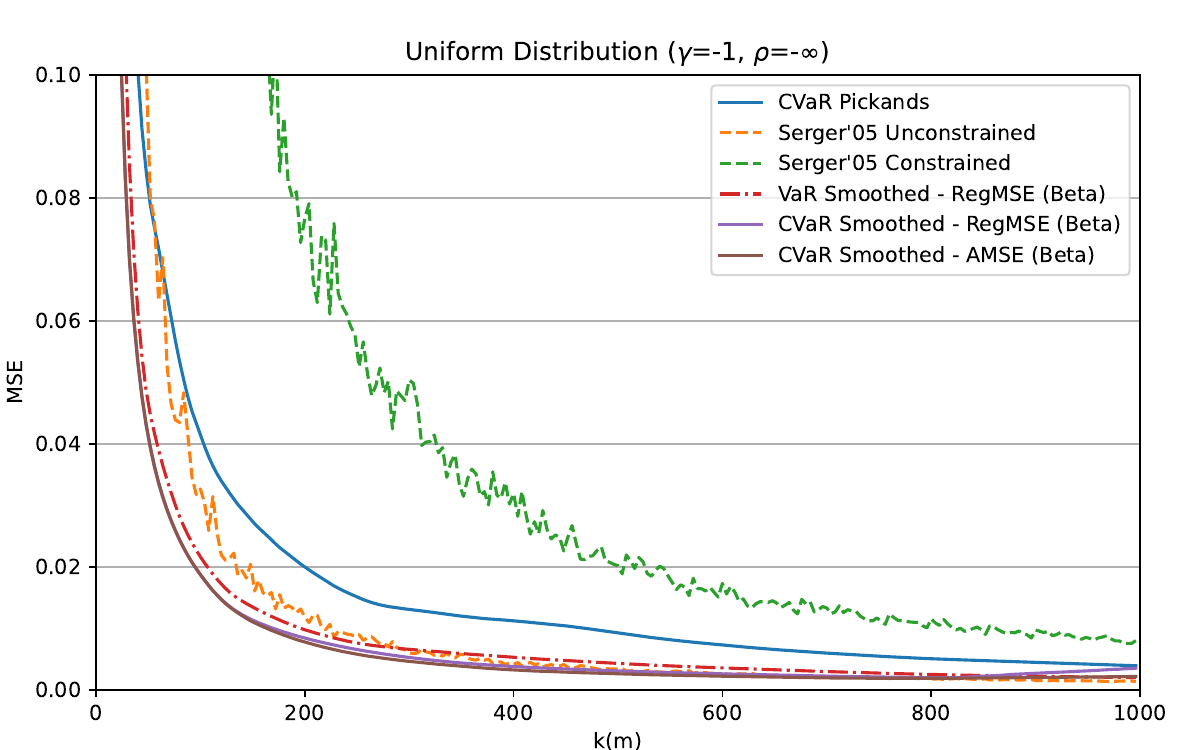}
    \end{subfigure}\hfill
    \begin{subfigure}{0.5\textwidth}
        \includegraphics[width=\linewidth]{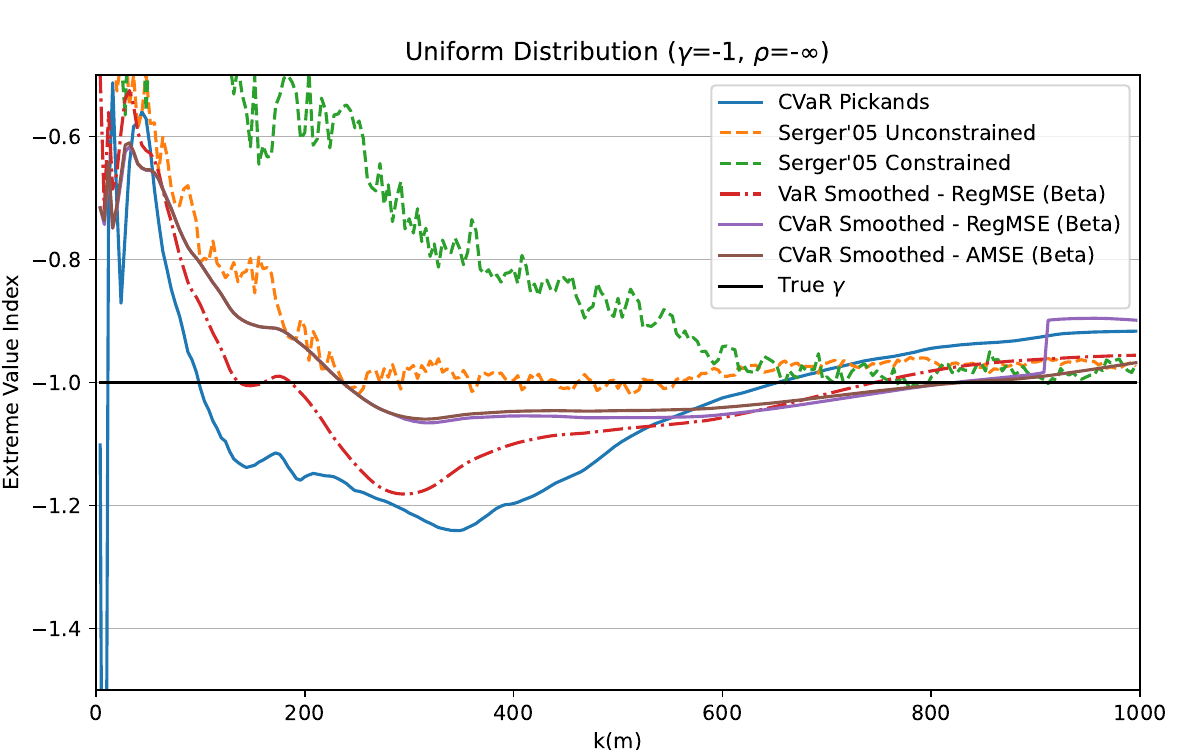}
    \end{subfigure}
    \caption{The samples are from the uniform distribution, and others are the same as Figure \ref{fig: burr0.25}}
    \label{fig: unif}
\end{figure}

\bibliography{myref}
\end{document}